\documentclass[11pt]{amsart}

\usepackage[hmargin=0.8in,twosideshift=0in,height=8.6in]{geometry}
\usepackage{amssymb,amsthm}
\usepackage{delarray,verbatim}
\usepackage{natbib}
\renewcommand{\cite}{\citeyearpar}
\usepackage{ifpdf}
\ifpdf
\usepackage[pdftex]{graphicx}
\else
\usepackage[dvips]{graphicx}
\fi

\usepackage{ifpdf}
\usepackage{color}
\definecolor{webgreen}{rgb}{0,.5,0}
\definecolor{webbrown}{rgb}{.8,0,0}
\definecolor{emphcolor}{rgb}{0.95,0.95,0.95}

\usepackage{hyperref}
\hypersetup{%
          colorlinks=true,
          linkcolor=webbrown,
          filecolor=webbrown,
          citecolor=webgreen,
          breaklinks=true}
\ifpdf
\hypersetup{pdftex,
            pdfstartview=FitH, %%Fit, FitB, FitH
            bookmarksopen=true,
            bookmarksnumbered=true
}
\else
\hypersetup{dvips}
\fi

\renewcommand{\theequation}{\thesection. \arabic{equation}}

\numberwithin{equation}{section}
\numberwithin{figure}{section}
\newtheorem{proposition}{Proposition}[section]
\newtheorem{coro}{Corollary}[section]
\newtheorem{remark}{Remark}[section]
\newtheorem{lemma}{Lemma}[section]
\newtheorem{example}{Example}[section]

\renewcommand{\S}{\mathcal{S}}
\newcommand {\ME}{\mathbb{E}^{x,i}}

\newcommand {\R}{\mathbb{R}}
\newcommand {\Fb}{\mathbb{F}}

\newcommand {\F}{\mathcal{F}}
\newcommand {\A}{\mathcal{A}}

\newcommand {\p}{\mathbb{P}}

\newcommand {\E}{\mathbb{E}}

\newcommand{\ttau}{\widetilde{\tau}}
\newcommand{\N}{\mathbb{N}}
\newcommand{\eps}{\varepsilon}
\newcommand{\diff}{{\rm d}}

\title[On the One-Dimensional Optimal Switching Problem]{On the One-Dimensional Optimal Switching Problem}
\author{Erhan Bayraktar }
\address[E. Bayraktar]{Department of
  Mathematics, University of Michigan, Ann Arbor, MI 48109}
\email{erhan@umich.edu}
\thanks{E. Bayraktar is supported in part by the National Science
Foundation. M. Egami is supported in part by Grant-in-Aid for
Scientific Research (C) No. 20530340, Japan Society for the
Promotion of Science.  An earlier version of this article is
available at Arxiv, see Bayraktar and Egami
\cite{bayraktaregami-2007}. We thank Savas Dayanik for his feedback
in the early stage of this work.}
\author{Masahiko Egami}
\address[M. Egami]{ Graduate School of Economics,
Kyoto University, Sakyo-Ku, Kyoto, 606-8501, Japan }
\email{egami@econ.kyoto-u.ac.jp}


\begin{document}

\begin{abstract}\noindent
We explicitly solve the optimal switching problem for
one-dimensional diffusions by directly employing the dynamic programming principle and the excessive characterization of the value function. The shape of the value function and the smooth fit principle then can be proved using the properties of concave functions.
\end{abstract}

\maketitle
\section{Introduction}  Stochastic \emph{optimal switching} problems (or \emph{starting and
stopping} problems) are important subjects both in mathematics and
economics. Switching problems were introduced into the study of \emph{real options} by Brennan and Schwarz
\cite{BS1985}  to determine the manager's optimal decision making  in resource extraction problems, and
 by Dixit \cite{D1989} to analyze production facility problems.
A switching problem in the case of a resource extraction problem can be described as follows:   The
controller monitors the price of natural resources and wants to optimize her profit by operating an extraction facility in an optimal way. She can
choose when to start extracting this resource and when to
temporarily stop doing so, based upon price fluctuations she
observes.  The problem is concerned with finding an optimal
starting/stopping (switching) policy and the corresponding value function.

There has been many recent developments in understanding the nature
of the optimal switching problems. When the underlying state
variable is geometric Brownian motion and for some special
reward/cost structure Brekke and {\O}ksendal \cite{BO1994},
Duckworth and Zervos \cite{DZ2001}, Zervos \cite{Zer2003} apply a
verification approach for solving the variational inequality
associated with the optimal switching problem. By using a viscosity solution approach, Pham and Ly Vath
\cite{pham05} generalize the previous results by solving the optimal
switching problem for more general reward functions. They do not
assume a specific form but only H\"{o}lder continuity of the reward
function. In contrast, our aim is
to obtain general results that applies to all one-dimensional
diffusions (in some switching problems a mean reverting process
might be more reasonable model for the underlying state process). Also, we will not assume the H\"{o}lder
continuity of the running reward function.

The verification approach applied in the above papers is indirect in the sense
that one first conjectures the form of the value function and the
switching policy and next verifies the optimality of the candidate
function by proving that the candidate satisfies the variational
inequalities.  In finding the specific form of the candidate
function, appropriate boundary conditions, including the smooth-fit
principle, are employed.  This formation shall lead to a system of
non-linear equations that are often hard to solve and the existence
of the solution to these system of equations is difficult to prove.  Moreover,
this indirect solution method is specific to the underlying process
and reward/cost structure of the problem.  Hence a slight change in
the original problem often causes a complete overhaul in the highly
technical solution procedures.

Our solution method is direct in the sense that we work with the
value function itself. First we characterize the value function
as the solution of two coupled optimal stopping problems. In other words we prove a dynamic programming principle. A proof of a dynamic programming principle for switching problems was given by Tang and Yong \cite{MR1306930} assuming a H\"{o}lder continuity condition on the reward function.
We give a new proof using a sequential approximation method (see Lemma~\ref{eq:main-result-1} and Proposition~\ref{prop:representation}) and avoid making this assumption. The properties of the essential supremum and optimal stopping theory for Markov processes play a key role in our proof.
Second, we give a sufficient condition which guarantees that the switching regions hitting times of certain closed sets (see Proposition~\ref{prop:hitting-times}). Next, making use of our sequential approximation we show when the optimal switching problem reduces to an ordinary stopping problem (see Proposition~\ref{prop:degeneracy}).
Finally, in the non-degenerate cases we construct an explicit solution (see Proposition~\ref{prop:smooth-fit}) using the excessive
characterization of the value functions of optimal stopping problem
(which corresponds to the concavity of the value function after a
certain transformation) Dayanik and Karatzas \cite{DK2003} (also see Dynkin \cite{Dynkin}, Alvarez \cite{Alv2001,
alvarez4}), see Lemma~\ref{lem:nlemm}. In Proposition~\ref{prop:smooth-fit}, we see that the continuation regions do not necessarily have to be connected. We give two examples, one of which illustrates this point. In the next example, we consider an problem in which the underlying state variable is an
Ornstein-Uhlenbeck process.

It is worth mentioning the work of Pham \cite{pham06}, which
provides another direct method to solve optimal switching problems
through the use of viscosity solution technique. Pham shows that the
value function of the optimal switching problem is continuously
differentiable and is the classical solution of its
quasi-variational inequality under the assumption that the reward
function is Lipschitz continuous. Johnson and Zervos \cite{JZ09}, on the other hand, by using a verification theorem, determine sufficient conditions that guarantee that the problem has connected continuation regions or is degenerate (see Section 5 and Theorem 7 of that paper).
A somewhat related problem to the
optimal switching problem we study here is the infinite horizon
optimal multiple stopping problem of Carmona and Dayanik
\cite{CD2003}, which was introduced to give a complete mathematical analysis of energy swing contracts. This problem is posed in the context of pricing
American options when the holder of the option has multiple  $n$
exercise rights. To make the problem non-trivial it is assumed that
the holder chooses the consecutive stopping times with a strictly
positive break period (otherwise the holder would use all his rights
at the same time). It is difficult to explicitly determine the
solution and Carmona and Dayanik describe a recursive algorithm to
calculate the value of the American option. In the switching
problems, however, there are no limits on how many times the
controller can switch from one state to another and one does not
need to assume a strictly positive break period. Moreover, we are
able to construct explicit solutions. Other related works include, Hamad\`{e}ne and Jeanblanc
\cite{HJ2004}, which analyzes a finite time horizon optimal switching problem with a general adapted observation process using the recently developed theory of reflected stochastic backward differential equations.
Carmona and Ludkovski \cite{CL2005} focus on a numerical resolution
based on Monte-Carlo regressions. Recently an interesting connection between the singular and the switching problems was given by Guo and Tomecek \cite{MR2373476}.

The rest of the paper is organized as follows:  In Section 2.1 we define the optimal switching problem. In Section 2.2 we study the problem in which the controller only can switch finitely many times. Using the results of Section 2.2, in Section 2.3 we give a characterization of the optimal switching problem as two coupled optimal stopping problems. In Section 2.4, we show that the \emph{usual} hitting times of the stopping regions are optimal. In Section 2.5 we give an explicit solution.
In Section 2.6 we give two examples illustrating our solution.

\section{The Optimal Switching Problem}
\subsection{Statement of the Problem}
Let $(\Omega, \F, \p)$ be a complete probability space hosting a Brownian motion $W=\{W_t; t\geq
0\}$. Let $\mathbb{F}=\left(\F_t\right)_{t \geq 0}$ be natural filtration of $W$. The controlled stochastic processes, $X$ with state space $(c,d)$ ($-\infty \leq c<d \leq \infty$), is a continuous process, which is defined as the solution of
\begin{equation}\label{eq:sde}
dX_t=\mu(X_t, I(t))dt + \sigma(X_t, I(t))dW_t, \quad X_0=x,
\end{equation}
in which the right-continuous switching process $I$ is defined as
\begin{equation}
I(t)=I_0 1_{\{t<\tau_1\}}+I_1 1_{\{\tau_1 \leq t <\tau_2\}}+\cdots+I_{n}1_{\{\tau_n \leq t <\tau_{n+1}\}}+\cdots
\end{equation}
where $I_{i} \in \{0,1\}$ and $I_{i+1}=1-I_i$ for all $i \in \mathbb{N}$.
Here, the sequence $(\tau_n)_{n \geq 1}$ is an increasing sequence of $\mathbb{F}$-stopping times with $\lim_{n \rightarrow \infty}\tau_n=\tau_{c,d}$, almost surely (a.s.). Here, $\tau_{c,d} \triangleq \inf \{t \geq 0: X_t=c \;\text{or}\; X_t=d\}$. The stopping time $\tau_{c,d}=\infty$ when both $c$ and $d$ are natural boundaries. We will denote the set of such sequences by $\S$. We will assume that the boundaries are either absorbing or natural.

 We are going to measure the performance of a strategy \begin{equation*}
T=(\tau_1,\tau_2\cdots,\tau_n,\cdots)
\end{equation*}
by
\begin{equation}\label{eq:problem}
J^T(x,i)=\E^{x,i}\left[\int_0^{\tau_{c,d}}e^{-\alpha
s}f(X_s,I_s)ds-\sum_{j} e^{-\alpha
\tau_j}H(X_{\tau_{j}}, I_{j-1},I_j)\right],
\end{equation}
in which
 $H: (c,d)\times \{0,1\}^2\rightarrow \mathbb{R}$ is the immediate benefit/cost of switching from $I_{j-1}$ to $I_{j}$. We assume that $H$ is continuous in its first variable and
 \begin{equation}\label{eq:ass-H}
 |H(x,i,1-i)| \leq C_{H}(1+|x|) \quad \text{for $x,y \in (c,d)$ and $i \in\{0,1\}$},
 \end{equation}
 for some strictly positive constants $C_{H}<\infty$. Moreover, we assume that
 \begin{equation}\label{eq:asspH}
 H(x,0,1)+H(x,1,0)>0.
\end{equation}
We also assume that the running benefit $f: (c,d) \times \{0,1\}\rightarrow \R$ is a continuous function and
satisfies
the linear growth condition:
\begin{equation}\label{eq:ass-f}
|f(x,i)| \leq C_{f} (1+|x|),
\end{equation}
for some strictly positive constant $C_{f}<\infty$.
This assumption will be crucial in what follows, for example it guarantees that
\begin{equation} \label{eq:f-condition}
\E^{x,i}\left[\int_0^{\tau_{c,d}} e^{-\alpha s}|f(X_s,I_s)|ds\right]<B (1+|x|),
\end{equation}
for some $B$,
if we assume that the discount rate is large enough, which will be a standing assumption in the rest of our paper
(see page 5 of Pham \cite{pham06}).

The goal of the switching problem then is to find
\begin{eqnarray}\label{eq:problem-2}
v(x,i)\triangleq\sup_{T \in \mathcal{S}}J^{T}(x,i), \quad  x \in (c,d), \; i \in\{0,1\},
\end{eqnarray}
and also to find an optimal $T \in \mathcal{S}$ if it exists.
\subsection{When the Controller Can Switch Finitely Many Times}\label{subsec:recursive}
For any $\Fb$ stopping time $\sigma$ let us define
\begin{equation}
S^n_{\sigma} \triangleq  \{(\tau_1,\cdots, \tau_n): \text{$\tau_i$ is an $\Fb$ stopping time for all $i \in \{1,\cdots,n\}$ and}\; \sigma \leq \tau_1 \leq \cdots \leq \tau_n<\tau_{c,d}\}.
\end{equation}
In this section, we will consider switching processes of the form
\begin{equation}\label{eq:I}
I^{(n)}(t)=I_0 1_{\{t<\tau_1\}}+ \cdots+I_{n-1} 1_{\{\tau_{n-1} \leq t <\tau_n\}}+I_{n} 1_{\{t \geq \tau_n\}},
\end{equation}
in which the stopping times $(\tau_1, \cdots, \tau_n) \in S_{0}^n$. By $X^{(n)}$ we will denote the solution of (\ref{eq:sde}) when we replace $I$ with $I^{(n)}$. So with this notation we have that
\begin{equation}\label{eq:X-0}
dX^{(0)}_t=\mu\left(X^{(0)}_t, I_0\right)dt + \sigma \left(X^{(0)}_t, I_0\right)dW_t, \quad X^{(0)}_0=x.
\end{equation}
We assume a strong solution to \eqref{eq:X-0} exits and that
\begin{equation}\label{eq:assum:dyn}
 |\mu(x,i)|+|\sigma(x,i)| \leq C (1+|x|),
\end{equation}
for some positive constant $C<\infty$, which guarantees the uniqueness of the strong solution.
We should note that
\begin{equation}
X^{(n)}_t=X^{(0)}_t, \; t \leq \tau_1\;;\; \cdots \; X^{(n)}_t=X_{t}^{(n-1)},\; t \leq \tau_{n}.
\end{equation}

The value function of the problem in which the controller chooses $n$ switches is defined as
\begin{equation}\label{eq:q-function}
q^{(n)}(x,i)\triangleq\sup_{(\tau_1,\cdots,\tau_n) \in \S_0^n}\ME\left[\int_{0}^{\tau_{c,d}}e^{-\alpha s}f(X^{(n) }_s,I^{(n)}_s)ds -\sum_{j=1}^n
e^{-\alpha \tau_j}H(X^{(n)}_{\tau_{j}}, I_{j-1},I_j)\right].
\end{equation}

We will denote the value of making no switches by $q^{(0)}$, which we define as
\begin{equation}
q^{(0)}(x,i)\triangleq \ME\left[\int_{0}^{\tau_{c,d}}e^{-\alpha s}f(X^{(0) }_s,i)ds\right],
\end{equation}
which is well defined due to our assumption in (\ref{eq:f-condition}).

Let $\tau_y$ be the first hitting time of $y\in
\mathcal{I}$ by $X^{(0)}$, and let $c\in\mathcal{I}$ be a fixed point of
the state space.  We set:
\begin{align} \nonumber
\begin{aligned}
    \psi_i(x) &=  \begin{cases}
                 \ME[e^{-\alpha\tau_c}1_{\{\tau_c<\infty\}}], & x\leq c, \\
                 1/\E^{c,i}[e^{-\alpha\tau_x}1_{\{\tau_x<\infty\}}],
                 &x>c,\end{cases}
\hspace{0.4cm}
    \varphi_i(x) &= \begin{cases}
                 1/\E^{c,i}\left[e^{-\alpha\tau_x}1_{\{\tau_x<\infty\}}\right], & x\leq c, \\
                 \ME[e^{-\alpha\tau_c}1_{\{\tau_c<\infty\}}],
                 &x>c,
     \end{cases}
\end{aligned}
\end{align}
 It should be
noted that $\psi_i(\cdot)$ and $\varphi_i(\cdot)$ consist of an
increasing and a decreasing solution of the second-order
differential equation $(\mathcal{A}_i-\alpha)u=0$ in $\mathcal{I}$
where $\A_i$ is the infinitesimal generator of $X^{(0))}$ when $I_0=i$ in (\ref{eq:X-0}).  They
are linearly independent positive solutions and uniquely determined
up to multiplication.  For the complete characterization of the functions
$\psi_i(\cdot)$ and $\varphi_i(\cdot)$ corresponding to various types of
boundary behavior see It\^{o} and McKean \cite{ito-mc}. For future use let us define the increasing functions
\begin{equation}
F_i(x) \triangleq \frac{\psi_i(x)}{\varphi_i(x)}, \quad \text{and} \quad G_i(x)\triangleq -\frac{\varphi_i(x)}{\psi_i(x)},\quad x \in (c,d),\; i \in \{0,1\}.
\end{equation}
In terms of the Wronskian of $\psi_i(\cdot)$ and $\varphi_i(\cdot)$ by
\begin{equation}
W_i(x):=\psi'_i(x)\varphi_i(x)-\psi_i(x)\varphi'_i(x).
\end{equation}
we can express $q^{(0)}(x,i)$ as
\begin{equation}\label{eq:R-lr}
\begin{split}
q^{(0)}(x,i)=&\left[\psi_i(x)-\frac{\psi_i(c)}{\varphi_i(c)}\varphi_i(x)\right]
\int_x^{d} \frac{
\left[\varphi_i(y)-\frac{\varphi_i(d)}{\psi_i(d)}\psi_i(y)\right]}{\sigma^2(y,i)
W_i(y)} f(y,i) dy
\\&+\left(\varphi_i(x)-\frac{\varphi_i(d)}{\psi_i(d)}\psi_i(x)\right) \int_{c}^{x} \frac{
\left[\psi_i(y)-\frac{\psi_i(l)}{\varphi(c)
}\varphi(y)\right]}{\sigma^2(y,i) W_i(y)} f(y,i) dy,
\end{split}
\end{equation}
 $x \in (c,d)$,
see e.g. Karlin and Taylor \cite{karlin-taylor} pages 191-204 and Alvarez \cite{alvarez2}
page 272.

Now,
consider the following sequential optimal stopping problems:
\begin{equation}\label{eq:w-function}
w^{(n)}(x,i)\triangleq\sup_{\tau\in \S_{0}^{1}}\ME\left[\int_0^\tau
e^{-\alpha s}f(X^{(0)}_s,i)ds + e^{-\alpha\tau}
\left(w^{(n-1)}(X^{(0)}_{\tau},1-i)-H(X^{(0)}_{\tau}, i, 1-i)\right)\right]
\end{equation}
where $w^{(0)}(x,i)=q^{(0)}(x,i)$, $x \in (c,d)$ and $i \in \{0,1\}$.

\begin{lemma}\label{eq:main-result-1}
For $n\in \mathbb{N}$, we have that $q^{(n)}(x,i)=w^{(n)}(x,i)$, for all $x \in (c,d)$ and $i \in \{0,1\}$. Moreover, $q^{(n)}$ is continuous in the $x$-variable.
\end{lemma}

\begin{proof}
See Appendix
\end{proof}

\subsection{Characterization of the Optimal Switching Problem as Two Coupled Optimal Stopping Problems}

Using the results of the previous section, here we will show that the optimal switching problem can be converted into two coupled optimal stopping problems.
\begin{coro}\label{lem:bef-main-result-2}
For all $x \in (c,d)$ and $i \in \{0,1\}$, the increasing sequence $(q^{(n)}(x,i))_{n \in \N}$ converges:
\begin{equation}
\lim_{n \rightarrow \infty}q^{(n)}(x,i)=v(x,i).
\end{equation}
Moreover, $v$ is continuous in the $x$-variable.
\end{coro}
\begin{proof}
Since $\S_{\sigma}^{n} \subseteqq \S_{\sigma}^{n+1} \subseteqq \S$, it follows that $(q^{(n)}(x,i))_{n \in \N}$ is a non-decreasing sequence and
\begin{equation}\label{eq:lim-q-n}
\lim_{n \rightarrow \infty}q^{(n)}(x,i) \leq v(x,i), \quad x \in (c,d), \; i \in \{0,1\}.
\end{equation}

Assume that $v(x,i)<\infty$.
Let us fix $x$ and $i$. For a given $\eps>0$, let $T=(\tau_1,\cdots, \tau_n, \cdots) \in \S$ be an $\eps$-optimal strategy, i.e.,
\begin{equation}\label{eq:eps-opt-cont}
J^{T}(x,i)\geq v(x,i)-\eps.
\end{equation}
Note that $T$ depends on $x$.
Now $T^{(n)} \triangleq (\tau_1,\cdots, \tau_n)\in \S_{0}^{(n)}$, and
\begin{equation}
X^{(n)}_t=X_t, \quad \text{and} \quad I^{(n)}_t=I_t, \quad t \leq \tau_n.
\end{equation}
Let $\tau_{c,d}$ be the smallest time that $X$ reaches $c$ or $d$, and $\tau_{c,d}^{(n)}$ be the smallest time $X^{(n)}$ reaches $c$ or $d$.

Since $\tau_{n} \rightarrow \tau_{c,d}$ as $n \rightarrow \infty$, almost surely,  it follows from the growth assumptions on $f$ and $H$ that
\begin{equation}\label{eq:f-eps}
\ME\left[\int_{\tau_n}^{\tau_{c,d}}e^{-\alpha t}|f(X_t,I_t)|dt+\int_{\tau_n}^{\tau^{(n)}_{c,d}}e^{-\alpha t}|f(X_t,I_t)|dt\right] <\eps,
\end{equation}
and
\begin{equation}\label{eq:H-eps}
\ME \left[\sum_{j>n}e^{-\alpha \tau_j}H(X_{\tau_j},I_{j-1},I_j)\right]<\eps.
\end{equation}
for large enough $n$. It follows from (\ref{eq:f-eps}) and (\ref{eq:H-eps}) that
\begin{equation}
\begin{split}
\liminf_{n \rightarrow \infty} J^{T^{(n)}}(x,i)&=\liminf_{n \rightarrow \infty}\ME\left[\int_{0}^{\tau_{c,d}}e^{-\alpha s}f(X^{(n) }_s,I^{(n)}_s)ds -\sum_{j=1}^n
e^{-\alpha \tau_j}H(X^{(n)}_{\tau_{j}}, I_{j-1},I_j)\right]
\\&\geq J^{T}(x,i)-2 \eps.
\end{split}
\end{equation}
Therefore, using (\ref{eq:eps-opt-cont}) we get
\begin{equation}
\liminf_{n \rightarrow \infty} q^{(n)}(x,i) \geq \liminf_{n \rightarrow \infty} J^{T^{(n)}}(x,i) \geq v(x,i)- 3 \eps.
\end{equation}
Since $\eps$ is arbitrary, this along with (\ref{eq:lim-q-n}) yields the proof of the corollary when $v(x,i)<\infty$.

When $v(x,i)=\infty$, then for each positive constant $B<\infty$, there exists $T \in \S$ such that
$J^{T}(x,i) \geq B$. Then, if we choose $T^{(n)}\in S_{0}^n$ as before with $\eps=1$, we get $J^{T^{(n)}} \geq B-2$, which leads to
\begin{equation}
 \liminf_{n \rightarrow \infty} q^{(n)}(x,i) \geq \liminf_{n \rightarrow \infty}J^{T^{(n)}} \geq B-2.
\end{equation}
Since $B$ is arbitrary, we have that
\begin{equation}
\lim_{n \rightarrow \infty}q^{(n)}(x,i)=\infty.
\end{equation}
It is clear from our proof that $q^{(n)}(x,i)$ converges to $v(x,i)$ locally uniformly. Since $x \rightarrow q^{(n)}(x,i)$ is continuous, the continuity of $x \rightarrow v(x,i)$ follows.

\end{proof}
The next  result shows that the optimal switching problem is equivalent to solving two coupled optimal stopping problems.
\begin{proposition}\label{prop:representation}
The value function of the optimal switching problem has the following representation for any $x \in (c,d)$ and $i \in \{0,1\}$:
\begin{equation}
v(x,i)=\sup_{\tau \in S_0^1} \ME \left[\int_0^\tau
e^{-\alpha s}f(X^{(0)}_s,i)ds + e^{-\alpha\tau}
\left(v(X^{(0)}_{\tau},1-i)-H(X^{(0)}_{\tau}, i, 1-i)\right)\right],
\end{equation}
which can also be written as
\begin{equation}\label{eq:v-opt-st}
v(x,i)=q^{(0)}(x,i)+\sup_{\tau \in S_0^1} \ME \left[ e^{-\alpha\tau}
\left(-q^{(0)}(X^{(0)}_{\tau},i)+v(X^{(0)}_{\tau},1-i)-H(X^{(0)}_{\tau}, i, 1-i)\right)\right],
\end{equation}
due to the strong Markov property of $X^{(0)}$.
\end{proposition}
\begin{proof}
First note that
\begin{equation}
w^{(n)}(x,i) \uparrow v(x,i), \quad \text{as $n \rightarrow \infty$},
\end{equation}
as a result of Proposition~\ref{eq:main-result-1} and Lemma~\ref{lem:bef-main-result-2}.
Therefore, it follows from (\ref{eq:w-function}) that
\begin{equation}
w^{(n)}(x,i) \leq \sup_{\tau \in S_0^1} \ME \left[\int_0^\tau
e^{-\alpha s}f(X^{(0)}_s,i)ds + e^{-\alpha\tau}
\left(v(X^{(0)}_{\tau},1-i)-H(X^{(0)}_{\tau}, i, 1-i)\right)\right].
\end{equation}

To obtain the opposite inequality let us choose $\ttau$ such that
\begin{equation}
\begin{split}
 \ME &\left[\int_0^{\ttau}
e^{-\alpha s}f(X^{(0)}_s,i)ds + e^{-\alpha\ttau}
\left(v(X^{(0)}_{\ttau},1-i)-H(X^{(0)}_{\ttau}, i, 1-i)\right)\right]
\\&\geq \sup_{\tau \in S_0^1} \ME \left[\int_0^\tau
e^{-\alpha s}f(X^{(0)}_s,i)ds + e^{-\alpha\tau}
\left(v(X^{(0)}_{\tau},1-i)-H(X^{(0)}_{\tau}, i, 1-i)\right)\right]-\eps.
\end{split}
\end{equation}
Then by the monotone convergence theorem
\begin{equation}
\begin{split}
&v(x,i)=\lim_{n \rightarrow \infty}w^{(n)}(x,i)
\\& \geq \lim_{n \rightarrow \infty} \ME\left[\int_0^{\ttau}
e^{-\alpha s}f(X^{(0)}_s,i)ds + e^{-\alpha\ttau}
\left(w^{(n-1)}(X^{(0)}_{\ttau},1-i)-H(X^{(0)}_{\ttau}, i, 1-i)\right)\right]
\\&=\ME \left[\int_0^{\ttau}
e^{-\alpha s}f(X^{(0)}_s,i)ds + e^{-\alpha\ttau}
\left(v(X^{(0)}_{\ttau},1-i)-H(X^{(0)}_{\ttau}, i, 1-i)\right)\right]
\\& \geq \sup_{\tau \in S_0^1} \ME \left[\int_0^\tau
e^{-\alpha s}f(X^{(0)}_s,i)ds + e^{-\alpha\tau}
\left(v(X^{(0)}_{\tau},1-i)-H(X^{(0)}_{\tau}, i, 1-i)\right)\right]-\eps.
\end{split}
\end{equation}
This proves the statement of the proposition.
\end{proof}

\begin{remark}
\begin{itemize}
\item[(i)]
It is clear that the result of the previous proposition holds even for finite horizon problems, which can be shown by making slight modifications (by setting the cost functions to be equal to zero after the maturity) to the proofs above.

\item[(ii)]Also, if there are more than two regimes the controller can choose from \eqref{eq:v-opt-st} can be modified to read
\begin{equation}\label{eq:v-opt-st-n}
v(x,i)=q^{(0)}(x,i)+\sup_{\tau \in S_0^1} \ME \left[ e^{-\alpha\tau}
\left(-q^{(0)}(X^{(0)}_{\tau},i)+\mathcal{M}v(X^{(0)}_{\tau},i)\right)\right],
\end{equation}
where
\begin{equation}
\mathcal{M} v(x,i)= \max_{j \in (\mathcal{I}-\{i\})}(v(x,j)-H(x, i, j)),
\end{equation}
and $\mathcal{I}$ is the set of regimes.
\end{itemize}
\end{remark}

\subsection{A Class of Optimal Stopping Times}
In this section, using the classical theory of optimal stopping times, we will show that hitting times of certain kind are optimal. We will first show that the assumed growth condition on $f$ and $H$ leads to a growth condition on the value function $v$, from which we can conclude that $v$ is finite on $(c,d)$.
\begin{lemma}\label{lem:2.3}
There exists a constant $C_v$ such that
\begin{equation}\label{lem:lin-grwth}
v(x,i) \leq C_{v}(1+|x|), \quad x \in (c,d), \; i \in \{0,1\}.
\end{equation}
In fact, the same holds for all $q^{(n)}$, $n \in \N$.
\end{lemma}
\begin{proof}
As in Pham \cite{pham06}
due to the linear growth condition on $b$ and $\sigma$, the process $X$ defined in (\ref{eq:sde}) satisfies the second moment estimate
\begin{equation}
\ME\left[X_{t}^2\right] \leq C e^{Ct} (1+|x|^2),
\end{equation}
for some positive constant $C$. Due to the linear growth assumption on $f$ we have that
\begin{equation}\label{eq:v-bdd}
\begin{split}
\ME \left[\int_0^{\infty}e^{-\alpha t}|f(X_t,I_t)|dt \right] & \leq C_f \, \ME \left[\int_0^{\infty}e^{-\alpha t}(1+|X_t|)dt\right] \\&\leq
\sqrt{C} C_f \int_0^{\infty}e^{-\alpha t}e^{Ct/2}(1+|x|)dt \leq C_v (1+|x|),
\end{split}
\end{equation}
for some large enough constant $C_v$. Here the second inequality follows from the Jensen's inequality and the fact that $\sqrt{(1+|x|)^2} \leq 1+|x|$. Also recall that we have assumed the discount factor $\alpha$ to be large enough.
(This is similar to the assumption in Pham~\cite{pham06}). Taking the supremum over $T \in \S$ in (\ref{eq:v-bdd}) we obtain that
\begin{equation}
v(x,i) \leq \sup_{T \in \S} \ME \left[\int_0^{\infty}e^{-\alpha t}|f(X_t,I_t)|dt \right]\leq C_v (1+|x|).
\end{equation}
The linear growth of $q^{(n)}$ can be shown similarly.
\end{proof}

\begin{proposition}\label{prop:hitting-times}
Let us define
\begin{equation}
\mathbf{\Gamma^{i}} \triangleq \{x \in (c,d): v(x,i)=v(x,1-i)-H(x,i,1-i)\}, \quad i \in \{0,1\}.
\end{equation}
Let us assume that $c=0$ and $d=\infty$ and the following one of the two hold:
\begin{enumerate}
\item $c$ is absorbing, and $d$ is natural,
\item Both $c$ and $d$ are natural.
\end{enumerate}
Then if for $i \in \{0,1\}$, $\lim_{x \rightarrow \infty}x/\psi_i(x)=0$, the stopping times
\begin{equation}\label{eq:opt-st-time}
\tau^{*,i} \triangleq \inf\{t \geq 0: X^{(0)}_t \in \Gamma^{i}\},
\end{equation}
are optimal. Note that $X^{(0)}$ in (\ref{eq:X-0}) depends on $I_0=i$, through its drift and volatility.
\end{proposition}

\begin{proof}
Let us prove the statement for Case 1.
First, we define
\begin{equation}\label{eq:l-d}
l^{i}_{d} \triangleq \lim_{x \rightarrow d}\frac{(v(x,1-i)-q^{(0)}(x,i)-H(x,i,1-i))^{+}}{\psi_i(x)}, \; i \in \{0,1\}.
\end{equation}
By Lemma~\ref{lem:2.3} $v$ and $q^{(0)}$ satisfy a linear growth condition. We assumed that $H$ also satisfies a linear growth condition. Therefore the assumption on $\psi_i$ guarantees that $l^{i}_{d}=0$, for $i \in \{0,1\}$. But then from Proposition 5.7 of Dayanik and Karatzas \cite{DK2003} the result follows.

For Case 2, we will also need to show that
\begin{equation}\label{eq:l-c-zero}
l^{i}_{c} \triangleq  \lim_{x \rightarrow c}\frac{(v(x,1-i)-q^{(0)}(x,i)-H(x,i,1-i))^{+}}{\varphi_i(x)}=0,
\end{equation}
and use Proposition 5.13 of Dayanik and Karatzas \cite{DK2003}. But the result is immediate since $v$, $q^{(0)}$ and $H$ are bounded in a neighborhood of $c=0$ and $\lim_{x \rightarrow c}\varphi_i(x)=\infty$, since $c$ is a natural boundary.

\end{proof}

\begin{remark}\label{rem:l-c-l-d}
If both $c$ and $d$ are absorbing it follows from Proposition 4.4 of Dayanik and Karatzas \cite{DK2003} that the stopping times in (\ref{eq:opt-st-time}) are optimal, since $H$, $q^{(0)}$ and $v$ are continuous. Also, observe that when $c$ is absorbing (\ref{eq:l-c-zero}) still holds since $v(c,i)=0$, $i \in\{0,1\}$. Similarly, when $d$ is absorbing $l^{i}_d$ in (\ref{eq:l-d}) is equal to zero.
\end{remark}

\begin{remark}\label{rem:emtyset}
Since $H(\cdot,i,1-i)+H(\cdot,1-i,i)$ is strictly positive, it can easily seen from the definition that $\mathbf{\Gamma^{0}} \cap \mathbf{\Gamma^{1}}=\emptyset$.
\end{remark}

\subsection{ Explicit Solutions} In this section, we let $c=0$ and $d=\infty$, and assume that $c$ is either natural or absorbing, and that $d$ is natural.

\begin{proposition}\label{prop:degeneracy}
Let us introduce the functions
\begin{equation}\label{eq:ass-h}
h_0(x)\triangleq q^{(0)}(x,1)-q^{(0)}(x,0)-H(x,0,1) \quad \text{and} \quad h_1(x)\triangleq q^{(0)}(x,0)-q^{(0)}(x,1)-H(x,1,0).
\end{equation}
\begin{itemize}
\item[(i)] If for all $x \in (0,\infty)$ we have that $h_0(x) \leq 0$ and $h_1(x) \leq 0$ , then $\mathbf{\Gamma^0}=\mathbf{\Gamma^1}=\emptyset$.
\item[(ii)] Let us assume that the dynamics of \eqref{eq:sde} do not depend on $I(t)$ (as a result $X^{(0)}=X$ and we will denote $\mathbb{E}^{x,1}=\mathbb{E}^{x,0}$ by $\mathbb{E}^x$).
Then $h_1(x) \leq 0$ for all $x \in (0,\infty)$ implies that $\mathbf{\Gamma^1}=\emptyset$. Similarly, if $h_0(x) \leq 0$ for all $x \in (0,\infty)$, then $\mathbf{\Gamma^0}=\emptyset$. (Observe that, in this case, the optimal switching problem reduces to an ordinary optimal stopping problem.)
\end{itemize}
\end{proposition}

\begin{proof}
\noindent \textbf{(i)\,}
For any $n \geq 1$, let us introduce
\begin{equation}\label{eq:defn-uold}
u^{(n)}(x,i) \triangleq w^{(n)}(x,i)-q^{(0)}(x,i), \quad x \in (0,\infty), \; i\in \{0,1\}.
\end{equation}
Using the strong Markov property of $X^{(0)}$ and \eqref{eq:w-function} we can write
\begin{equation}\label{eq:un}
u^{(n)}(x,i)=\sup_{\tau \in \S_{0}^1}\ME \left[e^{-\alpha \tau}\left(u^{(n-1)}(X^{(0)}_{\tau},1-i)+h_i(X^{(0)}_{\tau})\right)\right].
\end{equation}
Since $u^{(0)}(x,i)=0$ and $h_i(x) \leq 0$, it follows from \eqref{eq:un} that
\begin{equation*}
u^{(1)}(x,i)=0,  \quad x \in (0,\infty), \; i\in \{0,1\}.
\end{equation*}
If we assume that for $m \in \{1, \cdots, n-1\}$ we have that $u^{(m)}(x,i)=0$, $x \in (0,\infty)$,  $i\in \{0,1\}$; it follows  from \eqref{eq:un} that $u^{(m+1)}(x,i)=0$, $x \in (0,\infty)$,  $i\in \{0,1\}$. For a given $n$, we can carry out this induction argument to show that
\begin{equation*}
u^{(n)}(x,i)=0,  \quad x \in (0,\infty), \; i\in \{0,1\}.
\end{equation*}
Now using Lemma~\ref{eq:main-result-1} and Corollary~\ref{lem:bef-main-result-2} we have that
$v(x,i)=q^{(0)}(x,i)$, which yields the desired result.

\noindent \textbf{(ii)\,} We will only prove the first statement since the proof of the second statement is similar. As in the proof of (i) $h_1(x) \leq 0$ for all $x \in (0,\infty)$ implies that $u^{(1)}(\cdot,1)=0$. On the other hand,
\[
u^{(1)}(x,0)=\sup_{\tau \in \S_{0}^1} \mathbb{E}^{x}\left[e^{-\alpha \tau}h_0(X_{\tau})\right].
\]
Let us assume that for $m \in \{1, \cdots, n-1\}$
\[
u^{(m)}(x,1)=0, \quad \text{and that} \quad u^{(m)}(x,0)=\sup_{\tau \in \S_{0}^1} \mathbb{E}^{x}\left[e^{-\alpha \tau}h_0(X_{\tau})\right].
\]
Since $H(x,0,1)+H(x,1,0)>0$ we have that $h_0(x)+h_1(x)<0$, which in
turn implies that
\begin{equation}\label{eq:cruc-ineq}
u^{(m)}(x,0) \leq -\inf_{\tau \in \S_{0}^1} \mathbb{E}^{x}\left[e^{-\alpha \tau}h_1(X_{\tau})\right].
\end{equation}
Using \eqref{eq:cruc-ineq} we can write
\begin{equation}\label{eq:umineqaf}
\begin{split}
u^{(m+1)}(x,1)&=\sup_{\tau \in \S_{0}^1} \mathbb{E}^{x}\left[e^{-\alpha \tau}\left(u^{(m)}(X_{\tau},0)+h_1(X_{\tau})\right)\right]\\
&\leq \sup_{\tau \in \S_{0}^1} \mathbb{E}^{x}\left[e^{-\alpha \tau}u^{(m)}(X_{\tau},0) \right]+ \inf_{\tau \in \S_{0}^1} \mathbb{E}^{x}\left[e^{-\alpha \tau}h_1(X_{\tau})\right] \\
& = u^{(m)}(x,0)+\inf_{\tau \in \S_{0}^1} \mathbb{E}^{x}\left[e^{-\alpha \tau}h_1(X_{\tau})\right] \leq 0.
\end{split}
\end{equation}
The second equality in the above equation follows from the
assumption that the dynamics of \eqref{eq:sde} do not depend on
$I(t)$: Indeed, since the function $u^{(m)}$ is a value function of
an optimal stopping problem, then it is positive and $F$ concave
(see e.g. Proposition 5.11 of Dayanik and Karatzas \cite{DK2003}).
On the other hand,  by the same proposition of  Dayanik and Karatzas
\cite{DK2003} we note that $\sup_{\tau \in \S_{0}^1}
\mathbb{E}^{x}\left[e^{-\alpha \tau}u^{(m)}(X_{\tau},1-i) \right]$
is the smallest non-negative $F$ concave majorant of the function
$u^{(m)}$, which is non-negative and $F$-concave, it follows that
\[
\sup_{\tau \in \S_{0}^1} \mathbb{E}^{x}\left[e^{-\alpha \tau}u^{(m)}(X_{\tau},0) \right]=u^{(m)}(x,0).
\]
 Since the function $u^{(m+1)}$ is non-negative, it follows from \eqref{eq:umineqaf} that $u^{(m+1)} \equiv 0$. On the other hand,
 \[
 u^{(m+1)}(x,0)=\sup_{\tau \in \S_0^1} \mathbb{E}^{x}\left[e^{-\alpha \tau}\left(u^{(m)}(X_{\tau},1)+h_{0}(X_{\tau}) \right)\right]=\sup_{\tau \in \S_0^1} \mathbb{E}^{x}\left[e^{-\alpha \tau}h_{0}(X_{\tau}) \right]
 \]
  thanks to our induction hypothesis. Using this induction on $m$, we see that
 \[
u^{(n)}(x,1)=0, \quad \text{and that} \quad u^{(n)}(x,0)=\sup_{\tau \in \S_{0}^1} \mathbb{E}^{x}\left[e^{-\alpha \tau}h_0(X_{\tau})\right],
\]
for any given $n \in \mathbb{N}_+$. At this point applying Lemma~\ref{eq:main-result-1} and Corollary~\ref{lem:bef-main-result-2} we obtain that
\[
v(x,1)=q^{(0)}(x,1), \quad v(x,0)=\sup_{\tau \in \S_{0}^1}
\mathbb{E}^{x}\left[e^{-\alpha \tau}h_0(X_{\tau})\right]+q^{(0)}(x,
0).
\]
\end{proof}

In the rest of this section we will assume that the dynamics of \eqref{eq:sde} do not depend on $I(t)$ .
We will denote $F_i$ by $F$ and $G_i$ by $G$. We will also assume that $x \rightarrow
H(x,i,1-i)$, is continuously differentiable for $i \in \{0,1\}$.
\begin{proposition}
Let us define
\begin{equation}\label{eq:K0}
K_0(y) \triangleq \frac{h_0(F^{-1}(y))}{\varphi(F^{-1}(y))}, \quad y \in (0,\infty),
\end{equation}
and
\begin{equation}\label{eq:K1}
K_1(y) \triangleq \frac{h_1(G^{-1}(y))}{\psi(G^{-1}(y))}, \quad y \in (-\infty,0],.
\end{equation}
 Here $F^{-1}$
and $G^{-1}$ are functional inverses of $F$ and $G$,
respectively.
If $y \rightarrow K_0(y)$ is non-negative concave on $(0,\infty)$ then $\mathbf{\Gamma}^{0}=(0,\infty)$ and $\mathbf{\Gamma}^{1}=\emptyset$. Similarly, if  $y \rightarrow K_1(y)$ is non-negative concave on $(0,\infty)$ then $\mathbf{\Gamma}^{1}=(0,\infty)$ and $\mathbf{\Gamma}^{0}=\emptyset$.

\end{proposition}

\begin{proof}
We will only prove the first statement. In this case the function
$P_0$ in \eqref{eq:P0} is non-negative concave and its non-negative
concave majorant in \eqref{eq:V0} satisfies $V_0=P_0$, which implies
that $v(x,0)=v(x,1)-H(x,0,1)$ for all $x$. Therefore
$\mathbf{\Gamma}^{0}=(0,\infty)$. Thanks to Remark~\ref{rem:emtyset}
we necessarily have that $\mathbf{\Gamma}^{1}=\emptyset$.
\end{proof}

\begin{lemma}\label{lem:nlemm}
Let us define
\begin{equation}\label{eq:defn-u}
u(x,i) \triangleq v(x,i)-q^{(0)}(x,i), \quad x \in (0,\infty), \; i\in \{0,1\},
\end{equation}
and
\begin{equation}\label{eq:V0}
V_0(y) \triangleq \frac{u(F^{-1}(y),0)}{\varphi(F^{-1}(y))}, \quad y \in (0,\infty),
\end{equation}
\begin{equation}\label{eq:V1}
V_1(y) \triangleq \frac{u(G^{-1}(y),1)}{\psi(G^{-1}(y))}, \quad y \in (-\infty,0].
\end{equation}
Then the following statements hold:
\begin{itemize}
\item[(i)] $y \to V_0(y)$ and $y \to V_1(y)$ are the smallest positive concave majorants of
\begin{equation}\label{eq:P0}
P_0(y)\triangleq \frac{u(F^{-1}(y),1)}{\varphi(F^{-1}(y))}+K_0(y), \quad y \in (0,\infty),
\end{equation}
and
\begin{equation}
P_1(y)\triangleq \frac{u(G^{-1}(y),0)}{\psi(G^{-1}(y))}+K_1(y), \quad y \in (-\infty,0],
\end{equation}
respectively.
\item[(ii)] $V_0(0)=V_1(0)=0$.
\item[(iii)] $V_0$ is piecewise linear on $\{y \in \R_+: V_0(y)>P_0(y)\}$ and
\begin{equation}
\mathbf{C}_0:=\{y \in \R_+: V_0(y)=P_0(y)\} \subset \mathbf{K}_0:=\left\{y \in \R_+: \frac{\partial^2 K_0(y)}{\partial y^2}<0\right\}.
\end{equation}
Moreover, the function $P_0$ is concave on $\mathbf{K}_0$.
\item[(iv)]$V_1$ is piecewise linear on $\{y \in \R_{-}: V_1(y)>P_1(y)\}$ and
\begin{equation}
\mathbf{C}_1:=\{y \in \R_{-}: V_1(y)=P_1(y)\} \subset \mathbf{K}_1:=\left\{y \in \R_{-}: \frac{\partial^2 K_1(y)}{\partial y^2}<0\right\}.
\end{equation}
The function $P_1$ is concave on $\mathbf{K}_1$.
\end{itemize}
\end{lemma}

\begin{proof}
\textbf{(i).} It follows from (\ref{eq:v-opt-st}) that $u(\cdot,i)$ satisfies
\begin{equation}
u(x,i)=\sup_{\tau \in \S_{0}^1}\mathbb{E}^{x} \left[e^{-\alpha \tau}\left(u(X^{(0)}_{\tau},1-i)+q^{(0)}(X^{(0)}_{\tau},1-i)-q^{(0)}(X^{(0)}_{\tau},i)-H(X^{(0)}_{\tau},i,1-i)\right)\right].
\end{equation}
The statement follows from Theorem 16.4 of Dynkin \cite{Dynkin} (also see Proposition 5.11 of Dayanik and Karatzas \cite{DK2003}).

\noindent \textbf{(ii).} The result follows from (\ref{eq:l-c-zero})
and (\ref{eq:l-d}). We use Remark~\ref{rem:l-c-l-d} when $0$ is
absorbing.

\noindent\textbf{(iii), (iv).}  First, we want to
show that $P_0$ is concave on $\mathbf{K}_0$. It is enough to show
that
\[
 \frac{u(F^{-1}(y),1)}{\varphi(F^{-1}(y))}
\]
is concave on $\mathbf{K}_0$. But this can be shown using item (i).
Now the rest of the statement follows since $V_0$ is the smallest
concave majorant of $P_0$. Proof of (iv) follows similarly.
\end{proof}

In the next proposition we will give sufficient conditions under which the switching regions are connected and provide explicit solutions for the value function of the switching problem. We will also show that the value functions of the switching problem $x \rightarrow v(x,i)$, $i \in \{0,1\}$ are continuously differentiable under our assumptions.

In what follows we will assume that $H(x,i,1-i) \equiv H(i,1-i)$, $i \in \{0,1\}$.

\begin{proposition}\label{prop:smooth-fit}
Let us assume that the function $h_0$ and $h_1$ defined in Proposition~\ref{prop:degeneracy}
satisfy
\begin{equation}\label{eq:ass-1}
\lim_{x \to \infty}h_0(x)>0, \quad  \sup_{x>0}h_1(x)>0,
\end{equation}
and that  $\mathbf{K}_0=(M_0,\infty)$ for some $M_0>0$.
Then
\begin{equation}\label{eq:v-0-m}
v(x,0)=\begin{cases}
\beta_0 \psi(x)+q^{(0)}(x,0), & x \in (0,a) \\
\beta_1 \varphi(x)+q^{(0)}(x,1)-H(0,1), & x \in [a,\infty),
\end{cases}
\end{equation}
for some positive constants $a$, $\beta_0$ and $\beta_1$.
Moreover the following statements hold:
\begin{itemize}
\item[(i)] If $\mathbf{K}_1=(-\infty, -M_1)$ for some $M_1>0$, then $x \to v(x,1)$ has the following form
\begin{equation}\label{eq:v-1-m}
v(x,1)=
\begin{cases}
\beta_0 \psi (x)+q^{(0)}(x,0)-H(1,0), & x \in (0,b] \\
\beta_1 \varphi(x)+q^{(0)}(x,1), & x \in (b,\infty),
\end{cases}
\end{equation}
for a positive constant $b<a$.

\item[(ii)] If $\mathbf{K}_{1}=(-L,-N)$, for some $L, N>0$, $v(x,1)$ is of the form
\begin{equation} \label{eq:exp-v1}
v(x,1)=
\begin{cases}
\hat{\beta}_1 \psi(x)+q^{(0)}(x,1) & x \in (0,\tilde{b}) \\
\beta_0 \psi(x)+q^{(0)}(x,0)-H(1,0), & x \in [\tilde{b},c] \\
\widetilde{\beta}_1 \varphi(x)+q^{(0)}(x,1), & x \in (c,\infty).
\end{cases}
\end{equation}
for positive constant $c<a$ and $\tilde{b}<c$.
\end{itemize}

In both cases the value functions are continuously differentiable. As a result  the positive $a, b, \tilde{b}, c, \beta_0, \beta_1$, $\hat{\beta}_1$ and $\widetilde{\beta}_1$ can be determined from the continuous and the smooth fit conditions.
\end{proposition}

Before, we give the proof we will make two quick remarks.
\begin{remark}\label{rem:hs}
Note that $h_0(x)+h_1(x)<0$ for all $x \in \R_+$. So when $h_0(x)>0$, we have that $h_1(x)<0$.
The assumption that $\lim_{x \to \infty}h_0(x)>0$ ensures that the controller prefers state 0 to state 1, which can also be seen from the form of the value function $v(\cdot,0)$ in \eqref{eq:v-0-m}.
\end{remark}

\begin{remark}\label{rem:use-conv}
The following two identities can be checked to see when the functions $K_i$, $i \in \{0,1\}$ are concave:
\begin{equation}\label{eq:conc-iden1}
\frac{d^2 K_0(y)}{d y^2} \cdot \left(\A-\alpha\right) h_0(x) \geq 0, \quad \text{where} \quad y=F(x),
\end{equation}
and
\begin{equation}\label{eq:conc-iden2}
\frac{d^2 K_1(z)}{d z^2} \cdot \left(\A-\alpha\right) h_1(x) \geq 0,
\quad \text{where} \quad z=G(x),
\end{equation}
where $\A$ is the infinitesimal generator of $X$.

It follows from \eqref{eq:conc-iden1} and \eqref{eq:conc-iden2} that if $\frac{d^2 K_0(y)}{d y^2} \leq 0$, then $\frac{d^2 K_1(z)}{d z^2} \geq 0$. Let us prove this statement. First, due to \eqref{eq:conc-iden1}, $\frac{d^2 K_0(y)}{d y^2} \leq 0$ implies
\[
\left(\A-\alpha\right)(q^{(0)}(x,1)-q^{(0)}(x,0))+\alpha H(0,1)  \leq 0.
\]
As a result
\[
\left(\A-\alpha\right)(-q^{(0)}(x,1)+q^{(0)}(x,0))+\alpha H(1,0) \geq \alpha H(0,1)+\alpha H(1,0)>0,
\]
which implies $\frac{d^2 K_1(z)}{d z^2} \geq 0$ thanks to \eqref{eq:conc-iden2}.
\end{remark}

\vspace{0.1in}

\noindent \textbf{Proof of Proposition~\ref{prop:smooth-fit}.} The
proof is a corollary of Lemma~\ref{lem:nlemm}. Let us denote
$k:=\inf \mathbf{C}_0$. We will argue that $k<\infty$. If we assume
that $k=\infty$, then it follows that $\mathbf{\Gamma}^0=\R_+$ and
hence $v(x,0)=v(x,1)-H(0,1)=q^{(0)}(x,1)-H(0,1)$, where the last
identity follows from Remark~\ref{rem:emtyset}. Since $q^{(0)}(x,0)
\leq v(x,0)$, we obtain that
$h_1(x)=q^{(0)}(x,0)-q^{(0)}(x,1)-H(1,0) \leq - (H(0,1)+H(1,0))<0$
for all $x \in \R_+$. This contradicts our assumption on $h_1$ in
\eqref{eq:ass-1}.

Since $V_0(k)=P_0(k)$ and $y \to P_0(y)$ is concave on $(k, \infty)$ (by Lemma~\ref{lem:nlemm} (iii))
it follows that $\mathbf{C}_0=[k,\infty)$, due to the fact that $V_0$ is the smallest concave majorant of $P_0$. As a result, thanks also to Lemma~\ref{lem:nlemm}
(ii), we have that
\begin{equation}
V_0(y)=\begin{cases}
\alpha \,y & y \in (0,k)
\\ P_0(y) & y \in [k,\infty),
\end{cases}
\end{equation}
for some constants $\alpha >0$ that satisfies
\begin{equation}
\alpha \,k = P_0(k),
\end{equation}
which proves \eqref{eq:v-0-m}.\\

\noindent \textbf{(i)}
Similarly, if we let $l:=\sup \mathbf{C}_1$, then we have that this quantity is a finite negative number and that $\mathbf{C}_1=(-\infty,l)$. As a result,

\begin{equation}
V_1(y)=\begin{cases}
P_1(y) & y \in (-\infty,l],
\\ -\beta \, y & y \in (l,0],
\end{cases}
\end{equation}
for some constant $\beta>0$ that satisfies
\begin{equation}
-\beta \,l=P_1(l).
\end{equation}
 Next, we are going to determine $\alpha$, $\beta$, $k$ and $l$
making use of the fact that $V_0$ and $V_1$ are smallest
non-negative majorants of $P_0$ and $P_1$ further.  First observe
that $y \to K_0(y)$ is continuously differentiable since by
\eqref{eq:R-lr}, $x \rightarrow q^{(0)}(x,i)$ is continuously
differentiable. Second, by using (\ref{eq:defn-u}) and (\ref{eq:V0})
we obtain
\begin{equation}\label{eq:v0-pf}
u(x,0)=\begin{cases}
\alpha \psi(x), & x \in [0,F^{-1}(k)), \\
\beta \varphi(x)+q^{(0)}(x,1)-q^{(0)}(x,0)-H(0,1), & x \in [F^{-1}(k),\infty),
\end{cases}
\end{equation}
and
\begin{equation}\label{eq:v1-pf}
u(x,1)=\begin{cases}
\alpha \psi(x)+q^{(0)}(x,0)-q^{(1)}(x,1)-H(1,0), & x \in [0,G^{-1}(l)], \\
\beta \varphi(x), & x \in (G^{-1}(l),\infty).
\end{cases}
\end{equation}
It follows from Remark~\ref{rem:emtyset} that
\begin{equation}
G^{-1}(l) < F^{-1}(k).
\end{equation}
As a result, we have that
the function $x \rightarrow u(x,1)$ is differentiable on $(F^{-1}(k)-\epsilon,\infty)$, for some $\epsilon>0$.
Along with the differentiability of the function $K_0$, this observation yields that
\begin{equation}\label{eq:diif-V-0}
y \rightarrow P_0(y) \quad \text{is differentiable on} \quad [k-\delta_0,\infty),
\end{equation}
for some $\delta_0>0$.
Similarly, the differentiability of $x \rightarrow u(x,0)$ on $(0,G^{-1}(l))$ implies that
\begin{equation}\label{eq:diif-V-1}
y \rightarrow P_1(y) \quad \text{is differentiable on} \quad (-\infty,l+\delta_1],
\end{equation}
for some $\delta_1>0$.
>From (\ref{eq:diif-V-0}) and (\ref{eq:diif-V-1}) together with the fact that $V_0$ and $V_1$ are the smallest non-negative majorants of $P_0$ and $P_1$, we can determine $\alpha$, $\beta$, $k$ and $l$ from the following additional equations they satisfy
\begin{equation}
\begin{split}
\alpha=\frac{\partial P_0(y)}{\partial y}\bigg|_{y=k}, \quad \beta=-\frac{\partial P_1(y)}{\partial y}\bigg|_{y=l}.
\end{split}
\end{equation}
Using (\ref{eq:defn-u}) we can write the value functions $v(\cdot,i)$, $i \in \{0,1\}$ as
\begin{equation}\label{eq:v0-t-pf}
v(x,0)=\begin{cases}
\alpha \psi(x)+q^{(0)}(x,0), & x \in (0,F^{-1}(k)), \\
\beta \varphi(x)+q^{(0)}(x,1)-H(0,1), & x \in [F^{-1}(k),\infty),
\end{cases}
\end{equation}
and
\begin{equation}\label{eq:v1-t-pf}
v(x,1)=\begin{cases}
\alpha \psi(x)+q^{(0)}(x,0)-H(1,0), & x \in (0,G^{-1}(l)], \\
\beta \varphi(x)+q^{(0)}(x,1), & x \in (G^{-1}(l),\infty).
\end{cases}
\end{equation}
Now, a direct calculation shows that the left derivative and the right derivative of $x \rightarrow v(x,0)$ are equal at $x=F^{-1}(k)$. Similarly, one can show the same holds for the function $x \rightarrow v(x,1)$ at $x=G^{-1}(l)$. This completes the proof of (i). \\

\noindent \textbf{(ii)} Let us denote $s:=\inf \mathbf{C}_1$ and
$t:=\sup \mathbf{C}_1$. The function $P_1$ is concave on the
interval $[s,t]$, by Lemma~\ref{lem:nlemm} (iv).  Moreover, because
$V_1$ is the smallest non-negative majorant of $P_1$ it
follows that $\mathbf{C}_1=[s,t]$. Using the facts that $V_1$ is
piecewise linear on $\{y \in \R_-: V_1(y)>P_1(y)\}$, $V_1(0)=0$ and
$\lim_{y \to -\infty}K_1(y)<0$, the last being
equivalent to $\lim_{x \to 0}h_1(x)<0$ (this follows from our
assumption on $h_0$ in (\ref{eq:ass-1}), the
relation between $h_0$ and $h_1$ pointed out in Remark~\ref{rem:hs},
and our assumption on the set $\mathbf{K}_1$), we can write $V_1$ as
\begin{align*}
  V_1(y)=\begin{cases}
    \gamma_1, & y\in (-\infty, s),\\
    P_1(y), & y\in [s, t],\\
    -\widetilde{\beta}_1 y, &y\in (t, 0],
  \end{cases}
\end{align*}
from which \eqref{eq:exp-v1} follows. Note that $s \neq t$ since the
function $V_1$ is the smallest positive concave majorant of the
function $P_1$. The proof that the \emph{smooth fit property} is
satisfied at the boundaries follows similar line of arguments to the
proof of item (i). \hfill $\square$

\subsection{Examples}

\begin{example}\normalfont
In this example we will show how changing the switching costs we can
move from having one continuation regions (item (i) in
Proposition~\ref{prop:smooth-fit}) to disconnected continuation
regions (item (ii) in Proposition~\ref{prop:smooth-fit}). Let  the
running reward function in \eqref{eq:problem} be given by
$f_i(x)=k_ix^\gamma_i$ for $i=0, 1$ with $0<\gamma_0<\gamma_1<1$ and
$k_0>0, k_1\in \R_+$. We assume that dynamics of the underlying
state variable follow
\begin{equation*}
dX_t=mX_tdt + \beta X_tdW_t
\end{equation*}
where $m$ and $\beta$ are some given constants and
\begin{equation}\label{eq:ex1-condition}
m\gamma_1+\frac{\beta^2\gamma_1^2}{2}<\alpha.
\end{equation}

\textbf{Case 1. A Connected continuation region.} We assume that
$H(1, 0)<0$ and $H(0, 1)>0$. (Recall that $H(1,0)+H(0,1)>0$.)
Observe that our assumptions in Section 2.1, (\ref{eq:assum:dyn}),
(\ref{eq:ass-H}) and (\ref{eq:ass-f}) are readily satisfied.  In
what follows we will check the assumptions in item (i) of
Proposition~\ref{prop:smooth-fit} hold. First, let us obtain
functions, $\psi$, $\varphi$, $F$ , $G$, $q^{(0)}(\cdot,i)$, $i \in
\{0,1\}$, in terms of which we stated our assumptions. The
increasing and decreasing solutions of the ordinary differential
equation $(\A-\alpha)u=0$ are given by $\psi(x)=x^{\mu_+}$ and
$\varphi(x)=x^{\mu_-}$, where \[\mu_{+, -}=\frac{1}{\beta^2}\left(-m
+
\frac{1}{2}\beta^2\pm\sqrt{(m-\frac{1}{2}\beta^2)^2+2\alpha\beta^2}\right).\]
Note that under the assumption $\alpha>m$, we have $\mu_+>1$ and
$\nu_- <0$. Observe that $\lim_{x \rightarrow \infty} x/\psi=0$, $i
\in \{0,1\}$ (the main assumption of
Proposition~\ref{prop:hitting-times}).  It follows that
$F=x^{2\Delta/\beta^2}$ and $G=-x^{-2\Delta/\beta^2}$, in which
\[
\Delta=\sqrt{(m-\frac{1}{2}\beta^2)^2+2\alpha\beta^2}. %\quad
%\Delta_1=\sqrt{(m-\lambda-\frac{1}{2}\beta^2)^2+2\alpha\beta^2}.
\]
We can calculate $q^{(0)}(\cdot,i)$, $i \in \{0,1\}$ explicitly:
\begin{equation*}
q^{(0)}(x,i)=\ME\left[\int_0^\infty e^{-\alpha s}f_i(X^{(0)}_s)
\diff s\right]=\frac{k_ix^{\gamma_i}}{C_i},
\end{equation*}
where
$C_i:=\alpha-(m\gamma_i-\frac{1}{2}\beta^2\gamma_i(1-\gamma_i))>0$
since $0<\gamma_i<1$ and $\alpha>m$.
On the other hand,
\[
h_0(x)=
\frac{k_1x^{\gamma_1}}{C_1}-\frac{k_0x^{\gamma_0}}{C_0}-H(0,1),
\]
and
\[
h_1(x)
=\frac{k_0x^{\gamma_0}}{C_0}-\frac{k_1x^{\gamma_1}}{C_1}-H(1,0).
\]
The limits
\begin{equation}
\lim_{x \rightarrow \infty}h_0(x) =\infty, \quad \lim_{x \rightarrow
0}h_1(x)=-H(1,0)>0.
\end{equation}
Hence \eqref{eq:ass-1} in Proposition~\ref{prop:smooth-fit} is
satisfied.

Let us show that $\mathbf{K}_0=(M_0,\infty)$ and $\mathbf{K}_1=(-\infty, -M_1)$ for some $M_0, M_1>0$.
Remark~\ref{rem:use-conv} will be used to achieve this final goal.
\begin{equation}
\left(\A-\alpha\right)
h_0(x)=\frac{k_1}{C_1}\left(m\gamma_1+\frac{\beta^2\gamma_1^2}{2}
-\alpha\right)x^{\gamma_1}+\alpha H(0,1)-\frac{k_0}{C_0}\left(m\gamma_0+\frac{\beta^2\gamma_0^2}{2}-\alpha\right)x^{\gamma_0},
\end{equation}
which is negative only for large enough $x$ (since we assumed
\eqref{eq:ex1-condition}). On the other hand,
\begin{equation}\label{eq:concave-ex1-K1}
\left(\A-\alpha\right) h_1(x)
=\frac{k_0}{C_0}\left(m\gamma_0+\frac{\beta^2\gamma_0^2}{2}-\alpha\right)x^{\gamma_0}-\frac{k_1}{C_1}\left(m\gamma_1+\frac{\beta^2\gamma_1^2}{2}-\alpha\right)x^{\gamma_1}
+\alpha H(1,0),
\end{equation}
which is negative only for small enough non-negative $x$.

Thanks to Proposition~\ref{prop:smooth-fit}
the value function $v(x,i)$ is given by
\begin{align} \nonumber
\begin{aligned}
   v(x, 0) &= \begin{cases}
                 \beta_0x^{\mu_{+}}+\frac{k_0}{M_0}x^\gamma_0, & x \in [0,a), \\
                 \beta_1
x^{\mu_-}+\frac{k_1}{M_1}x^\gamma_1-H(0,1),
                 &x \in [a,\infty),
   \end{cases};
  \hspace{0.2cm}
    v(x, 1) &=  \begin{cases}
                 \beta_0x^{\mu_{+}}+\frac{k_0}{M_0}x^\gamma_0-H(1,0), & x\in [0,b], \\
                 \beta_1
x^{\mu_-}+\frac{k_1}{M_1}x^\gamma_1,
                 &x \in [b,\infty),\end{cases}
\end{aligned}
\end{align}
in which the positive constants $\beta_0$, $\beta_1$, $a$ and $b$
can be determined from continuous and smooth fit conditions.
Figure~\ref{fig:ex1-connect} illustrates a numerical example. Note
that since the closing cost $C$ is negative, $v(x, 1)\ge v(x, 0)$
for all $x\in \R_+$.
\begin{figure}[h]
\label{fig:ex1-connect}
\begin{center}
\begin{minipage}{1\textwidth}
\centering{\includegraphics[scale=1]{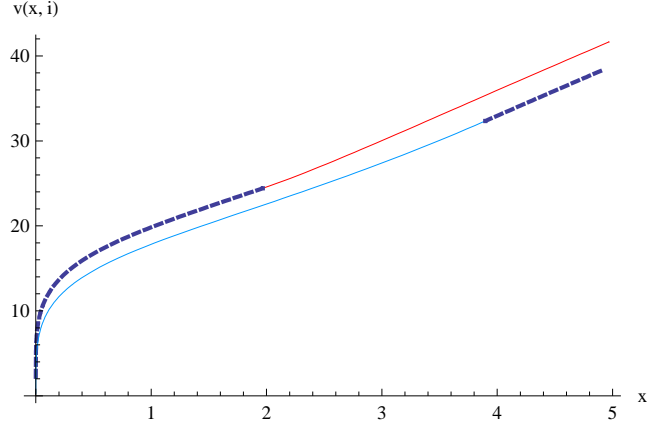}}
\end{minipage}
\caption{\small A numerical example illustrating Case 1.
Here, $(m, \beta, \alpha, L, C)=(0.01, 0.25, 0.1, 3, -2)$
and $(\gamma_0, \gamma_1, k_0, k_1)=(0.25, 0.75, 1.8, 1.2)$. The
unknown variables are determined to be $(a, b, \beta_0, \beta_1)=(3.8954, 1.9678,
0.416971, 11.3264)$. $v(x, 0)$ is plotted in a blue line on $x\in
(0, a]$ and in a dashed line on $x\in (a, \infty)$. $v(x, 1)$ is
plotted in a dashed line on $x\in (0, b]$ and in a red line on $x\in
(b, \infty)$.}
\end{center}
\end{figure}

\textbf{Case 2. Multiple continuation regions.} We will change the
value of switching from 1 to 0 and assume that it is positive, i.e.,
$H(1, 0)>0$ while we keep the assumption that $H(0, 1)>0$. Clearly,
$\sup_{x>0}h_1(x)>0$.  Moreover, the analysis of
(\ref{eq:concave-ex1-K1}) easily shows that $\left(\A-\alpha\right)
h_1(0)=C\alpha>0$ and $\lim_{x\rightarrow
\infty}\left(\A-\alpha\right) h_1(x)=\infty>0$. As a result
$\left(\A-\alpha\right) h_1(x)=0$ has two real roots. This fact
translates into the fact that $\mathbf{K}_1=(-L,-N)$ for some $L,
N>0$.
 See Figure 2.2 for the shape of
$K_1(y)$ for a particular set of parameters.

\begin{figure}[h]

\begin{center}
\begin{minipage}{0.45\textwidth}
\centering{\includegraphics[scale=0.75]{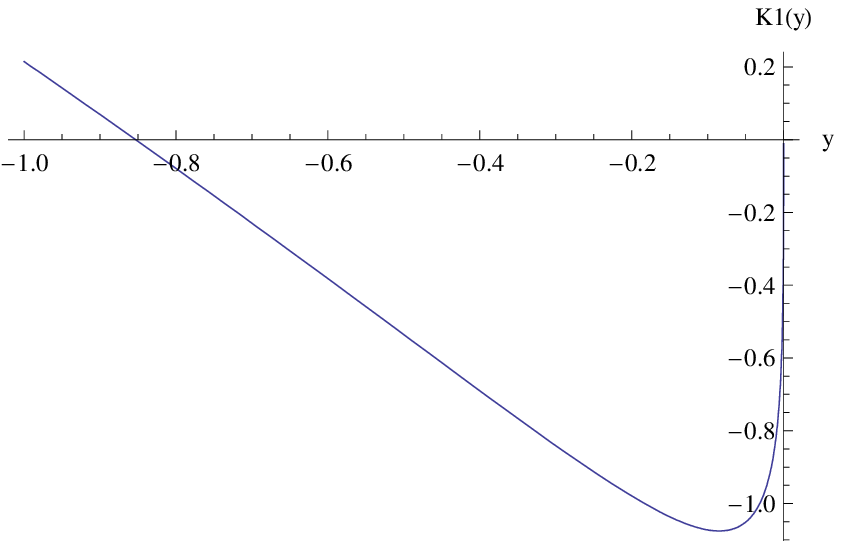}}\\
(a)
\end{minipage}
\begin{minipage}{0.45\textwidth}
\centering{\includegraphics[scale=0.75]{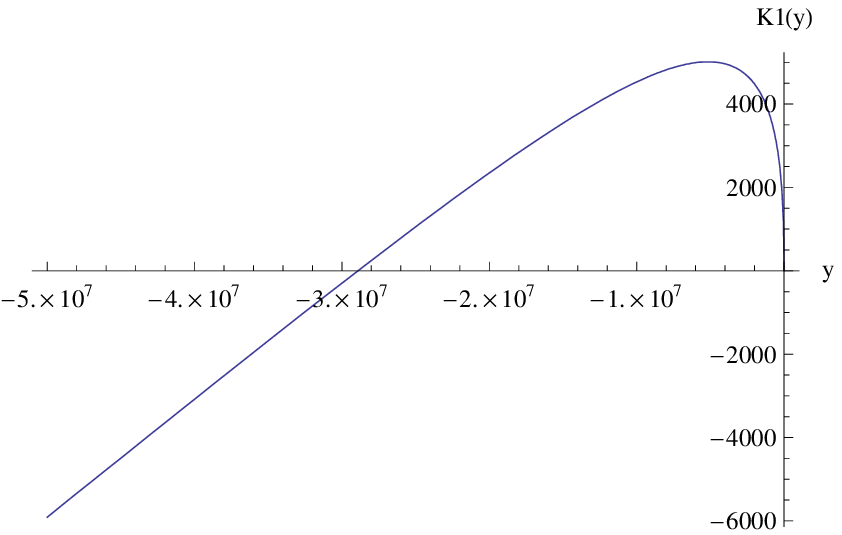}}\\
(b)
\end{minipage}
 \caption{\small A
numerical example illustrating Case 2. Here,
$(m, \beta, \alpha, H(0,1), H(1,0))=(0.01, 0.25, 0.1, 1, 5)$ and $(\gamma_0,
\gamma_1, k_0, k_1)=(0.25, 0.75, 1.8, 1.2)$. (a) $K_1(y)$ function
(a) in the neighborhood of the origin in the transformed space and
(b) in the large negative value.}
\end{center}
\label{fig:shape}
\end{figure}

Thanks to Proposition
\ref{prop:smooth-fit}, the value function $v(x,1)$ is given by \eqref{eq:exp-v1}.
Figure 2.3 displays the solution for a particular set of  parameters.
\begin{figure}[h]
\label{fig:2}
\begin{center}
\begin{minipage}{0.45\textwidth}
\centering{\includegraphics[scale=0.75]{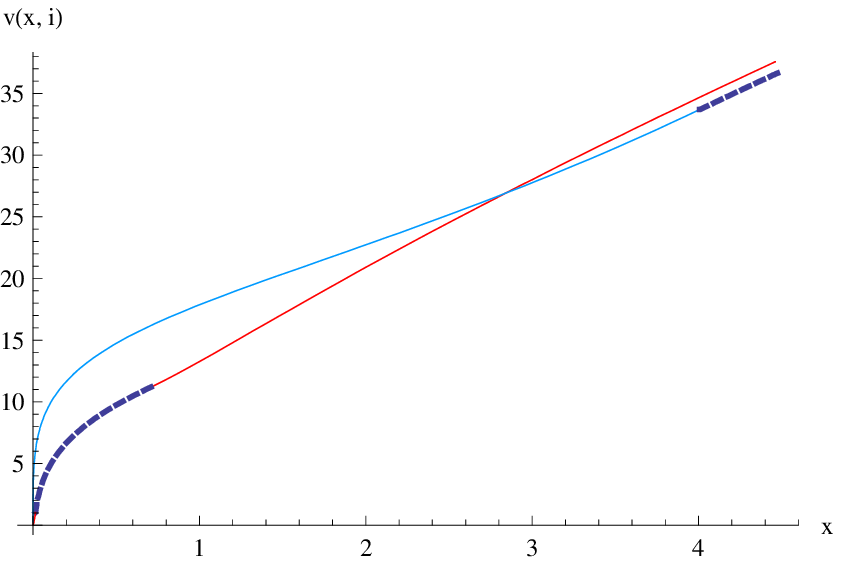}}\\
(a)
\end{minipage}
\begin{minipage}{0.45\textwidth}
\centering{\includegraphics[scale=0.75]{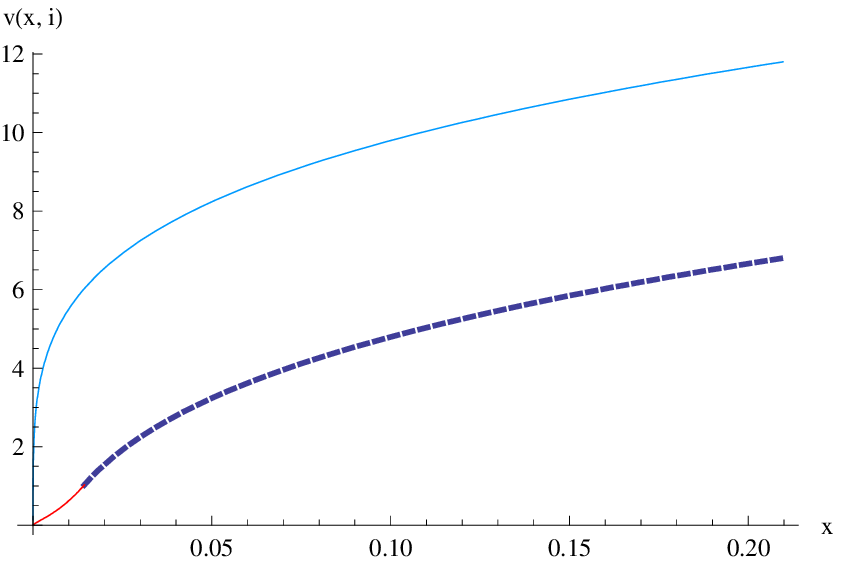}}\\
(b)
\end{minipage}
\caption{\small (a) A numerical example illustrating Case 2. Here,
$(m, \beta, \alpha, H(0,1), H(1,0))=(0.01, 0.25, 0.1, 1, 5)$ and
$(\gamma_0, \gamma_1, k_0, k_1)=(0.25, 0.75, 1.8, 1.2)$. The unknown
variables are determined to  be $(\beta_0, \widetilde{\beta}_1,
\hat{\beta}_1)=(0.450813, 1.06257, 5014.6)$ and $(a, \tilde{b},
c)=(4.00677, 0.0143517, 0.709694)$. $v(x, 0)$ is plotted in a blue
line on $x\in (0, a)$ and in a dashed line on $x\in [a, \infty)$.
$v(x, 1)$ is plotted in a red line in $x\in (0, \tilde{b})$ and $(c,
\infty)$.  It is plotted in a dashed line on $x\in [\tilde{b},c]$.
(b) To show the multiple continuation regions more clearly, we
magnify the left picture (a) near the origin.}
\end{center}
\end{figure}

\begin{example} \normalfont \textbf{Ornstein-Uhlenbeck process.}
The purpose of this exercise is to solve give an example of an optimal switching problem for a mean-reverting process. Let the dynamics of the state variable be given by
\begin{equation*} dX_t=\delta(m-X_t)\diff t+\sigma dW_t.
\end{equation*}
Let the value function be defined by
\begin{equation*}
v(x,i)=\sup_{T \in \S}\ME\left[\int_0^{\tau_0}
e^{-\alpha t}(X_t-K)I_tdt -
\sum_{\tau_i<\tau_0}e^{-\alpha\tau_i}H(X_{\tau_i},I_{i},I_{i+1})\right],
\end{equation*}
in which $\tau_0=\inf\{t>0:X_t=0\}$.
Our assumptions in Section 2.1, (\ref{eq:assum:dyn}), (\ref{eq:ass-H}) and (\ref{eq:ass-f}) are satisfied by our model.
Let us introduce
\begin{equation}\label{eq:t-phi-pis}
\tilde{\psi}(x) \triangleq e^{\delta
x^2/2}\mathcal{D}_{-\alpha/\delta}(-x\sqrt{2\delta})\quad\text{and}\quad
\tilde{\varphi}(x) \triangleq e^{\delta
x^2/2}\mathcal{D}_{-\alpha/\delta}(x\sqrt{2\delta}),
\end{equation}
where
$\mathcal{D}_\nu(\cdot)$ is the parabolic cylinder function; (see
Borodin and Salminen \cite{BS2002}(Appendices 1.24 and 2.9), which is given in terms of the Hermite function as
\begin{equation} \label{eq:Dft}
\mathcal{D}_\nu(z)=2^{-\nu/2}e^{-z^2/4}\mathcal{H}_\nu(z/\sqrt{2}),
\quad z\in\mathbb{R}.
\end{equation}
Recall that Hermite function $\mathcal{H}_\nu$ of degree $\nu$
and its integral representation
\begin{equation}\label{eq:Hermite}
\mathcal{H}_\nu(z)=\frac{1}{\Gamma(-\nu)}\int_0^\infty
e^{-t^2-2tz}t^{-\nu-1}dt, \quad \text{Re}(\nu)<0,
\end{equation}
(see for example, Lebedev \cite{L1972} (pages 284, 290)).
In terms of the functions in (\ref{eq:t-phi-pis}) the fundamental solutions of $(\A_0-\alpha)u=0$ and $(\A_1-\alpha)u=0$ are given by
\[
\begin{split}
\psi(x)&=\tilde{\psi}((x-m)/\sigma), \quad
\varphi(x)=\tilde{\varphi}((x-m)/\sigma).
% ,
%\\\psi_1(x)&=\tilde{\psi}((x-m+\lambda)/\sigma),
%\quad \varphi_1(x)=\tilde{\varphi}((x-m+\lambda)/\sigma).
\end{split}
\]
Observe that $\lim_{x \rightarrow \infty}x/\psi(x)=0$ (the main
assumption of Proposition~\ref{prop:hitting-times}). Since
$\E^x[X^{(0)}_t]=e^{-\delta t}x+m(1-e^{-\delta t})$, we have
\begin{equation}
q^{(0)}(x,0)=0 \quad \text{and} \quad q^{(0)}(x,
1)=\frac{x-m}{\delta+\alpha}+\frac{m-K}{\alpha}.
\end{equation}
Note that the limits of the functions
\[
h_0(x)=
q^{(0)}(x,1)-q^{(0)}(x,0)-H(0,1)=\frac{x-m}{\delta+\alpha}+\frac{m-K}{\alpha}-H(0,1),
\]
and
\[
h_1(x) =-\left(\frac{x-m}{\delta+\alpha}+\frac{m-K}{\alpha}+H(1,0)\right)
\]
are given by
\[
\lim_{x \rightarrow \infty}h_0(x)=\infty, \quad \text{and} \quad
\lim_{x \rightarrow
0}h_1(x)=-\left(\frac{-m}{\delta+\alpha}+\frac{m-K}{\alpha}+H(1,0)\right),
\]
When $\frac{K}{\alpha}-\frac{\delta m}{\alpha(\alpha+\delta)}>H(1,0)$,
then $\lim_{x \rightarrow
0}h_1(x)>0$.

Let us show that $\mathbf{K}_0=(M_0,\infty)$ and $\mathbf{K}_1=(-\infty, -M_1)$ for some $M_0, M_1>0$.
 For this purpose we
will again use Remark~\ref{rem:use-conv}.
\begin{equation}
\left(\A-\alpha\right) h_0(x)=\left(\A-\alpha\right)
\left(\frac{x-m}{\delta+\alpha}+\frac{m-K}{\alpha}-H(0,1)\right)=
-x+\frac{\delta m}{\delta+\alpha}+K- \alpha H(0,1),
\end{equation}
which implies that the function $K_0$ is concave only $(M_0,\infty)$, for some $M_0>0$. On the other hand,
\begin{equation}
\left(\A-\alpha\right) h_1(x)
=-\left(\A-\alpha\right)\left(\frac{x-m}{\delta+\alpha}+\frac{m-K}{\alpha}+H(1,0)\right)=
x-\frac{\delta m}{\delta+\alpha}-K+\alpha H(1,0),
\end{equation}
which implies that $K_1(\cdot)$ is concave only on $(-\infty,-M_1)$ for some $M_1>0$. Now, as a result of Proposition~\ref{prop:smooth-fit}, we have that
\begin{align} \nonumber
\begin{aligned}
   v(x, 0) &= \begin{cases}
                 \hat{v}_0(x), & x\in (0,a), \\
                 \hat{v}_1(x)-L,
                 &x \in [a,\infty),
     \end{cases};
  \hspace{0.2cm}
    v(x, 1) &=  \begin{cases}
                 \hat{v}_0(x)-C, & x\in (0,b], \\
                 \hat{v}_1(x),
                 &x\in(b,\infty),\end{cases}
\end{aligned}
\end{align}
in which
\begin{align*}
\hat{v}_0(x)&=\beta_0 (\psi(x)-F(0)\varphi(x))+q^{(0)}(x,0) \\
&=\beta_0e^{\frac{\delta}{2}\frac{(x-m)^2}{\sigma^2}}\left\{\mathcal{D}_{-\alpha/\delta}\left(-\left(\frac{x-m}{\sigma}\right)\sqrt{2\delta}\right)
-F(0)\mathcal{D}_{-\alpha/\delta}\left(\left(\frac{x-m}{\sigma}\right)\sqrt{2\delta}\right)\right\},
\end{align*}
and
\begin{align*}
\hat{v}_1(x)&=\beta_1\varphi_1(x)+q^{(0)}(x,1) \\
&=\beta_1
e^{\frac{\delta(x-m)^2}{2\sigma^2}}\mathcal{D}_{-\alpha/\delta}\left(\frac{(x-m)\sqrt{2\delta}}{\sigma}\right)
+\frac{x-m}{\delta+\alpha}+\frac{m-K}{\alpha}.
\end{align*}
The parameters, $a$, $b$, $\beta_0$ and $\beta_1$ can now be
obtained from continuous and smooth fit since we know that the value
functions $v(\cdot,i)$, $i \in \{0,1\}$ are continuously
differentiable. See Figure 2.4 for a numerical example.
\begin{figure}[h]\label{fig:3}
\begin{center}
\begin{minipage}{1\textwidth}
\centering {\includegraphics[scale=1]{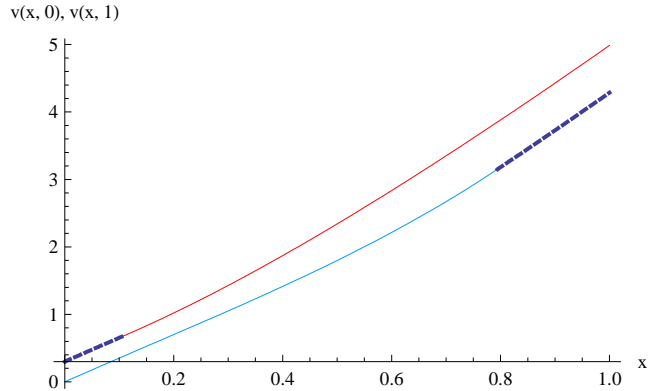}} \\
%(a)
\end{minipage}
%\begin{minipage}{0.45\textwidth}
%\centering \includegraphics[scale=0.75]{OUv1} \\
%(b)
%\end{minipage}
\caption{\small A numerical solution illustrating Example 2.2.
Here, $(m, \alpha, \sigma, \delta, K, H(0,1), H(1,0))=(0.5,
0.105, 0.35, 0.05, 0.4, 0.7, -0.3)$. The switching boundaries are
$a=0.1079$, $b=0.7943$. The other unknowns in (\ref{eq:v-0-m}) and
(\ref{eq:v-1-m}) are determined as $\beta_0=2.7057$ and
$\beta_1=1.4420$. $v(x, 0)$ is plotted in a blue line on $x\in (0,
a]$ and in a dashed black line on $x\in (a, \infty)$.  $v(x, 1)$ is
plotted in a dashed line on $x\in (0, b]$ and in a red line on $x\in
(b, \infty)$.}
\end{center}
\end{figure}
\end{example}
\end{example}

\appendix
\renewcommand{\thesection}{A}
\refstepcounter{section}
\makeatletter
\renewcommand{\theequation}{\thesection.\@arabic\c@equation}
\makeatother

\section{Proof of Lemma~\ref{eq:main-result-1}}

We will approximate the switching problem by iterating optimal stopping problems. This approach is motivated by Davis \cite{MR1283589} (especially the section on impulse control) and \O ksendal and Sulem \cite{oksendal-book-2}.
To establish our goal we will use the properties of the essential supremum (see Karatzas and Shreve \cite{kn:karat2}, Appendix A) and the optimal stopping theory for Markov processes in Fakeev \cite{Fakeev}. A similar proof, in the context of ``multiple optimal stopping problems", is carried out by Carmona and Dayanik \cite{CD2003}.

 For any $\Fb$ stopping time $\sigma$, let us define
 \begin{equation}
 Z^{(n)}_{\sigma}\triangleq \text{ess sup}_{(\tau_1, \cdots, \tau_n) \in \S_{\sigma}^n}\ME\left[\int_{\sigma}^{\tau_{c,d}}e^{-\alpha s}f(X^{(n) }_s,I^{(n)}_s)ds -\sum_{j=1}^n
e^{-\alpha \tau_j}H(X^{(n)}_{\tau_{j}}, I_{j-1},I_j)\bigg | \F_{\sigma}\right] \geq 0,
 \end{equation}
for $n \geq 1$, and
\begin{equation}
Z^{(0)}_{\sigma} \triangleq \ME\left[\int_{\sigma}^{\tau_{c,d}}e^{-\alpha s}f(X^{(0) }_s,I_0)ds\bigg | \F_{\sigma}\right] \geq 0.
\end{equation}
We will perform the proof of the lemma in four steps.

\noindent \textbf{Step 1}. If we can show that the family
\begin{equation}\label{eq:family}
{\mathcal{Z}} \triangleq \left\{\E^{x,i}\left[\int_{\sigma}^{\tau_{c,d}}e^{-\alpha s}f(X^{(n) }_s,I^{(n)}_s)ds -\sum_{j=1}^n
e^{-\alpha \tau_j}H(X^{(n)}_{\tau_{j}}, I_{j-1}, I_j)\bigg | \F_{\sigma}\right]: (\tau_1,\cdots,\tau_n) \in \S_{\sigma}^n\right\},
\end{equation}
 is directed upwards, it follows from the properties of the essential supremum (see Karatzas and Shreve (1998,Appendix A)) that for all $n \in \N$
 \begin{equation}
 Z^{(n)}_{\sigma} =\lim_{k \rightarrow \tau_{c,d}} \uparrow \ME\left[\int_{\sigma}^{\tau_{c,d}}e^{-\alpha s}f(X^{(n),k }_s,I^{(n),k}_s)ds -\sum_{j=1}^n
e^{-\alpha \tau_j}H(X^{(n),k}_{\tau_{j}^k}, I_{j-1}, I_j)\bigg | \F_{\sigma} \right]
 \end{equation}
 for some sequence $\left\{\left(\tau_1^{k}, \cdots \tau_n^{k}\right)\right\}_{k \in \mathbb{N}} \subset \S_{\sigma}^n$. Here, $X^{(n),k}$ is the solution of (\ref{eq:sde}) when we replace $I$ by $I^{(n),k}$ which is defined as
 \begin{equation}
 I^{(n),k}(t) \triangleq I_0 1_{\{t<\tau^k_1\}}+ \cdots+I_{n-1} 1_{\{\tau_{n-1}^k \leq t <\tau_n^k\}}+I_{n} 1_{\{t \geq \tau_n^k\}}.
 \end{equation}

We will now argue that  (\ref{eq:family}) is directed upwards (see Karatzas and Shreve \cite{kn:karat2} Appendix A for the definition of this concept ): For any $(\tau^1_1, \cdots, \tau^1_n),  (\tau^2_1, \cdots, \tau^2_n)\in \S_{\sigma}^n$, let us define the event
\begin{equation}
\begin{split}
A \triangleq &\Bigg\{\ME\bigg[\int_{\sigma}^{\tau_{c,d}}e^{-\alpha s}f(X^{(n),1 }_s,I^{(n),1}_s)ds  -\sum_{j=1}^n
e^{-\alpha \tau^1_j}H(X^{(n),1}_{\tau_{j}^1}, I_{j-1}, I_j)\bigg | \F_{\sigma} \bigg] \\ &\geq\ME\bigg[\int_{\sigma}^{\tau_{c,d}}e^{-\alpha s}f(X^{(n),2 }_s,I^{(n),2}_s)ds -\sum_{j=1}^n
e^{-\alpha \tau_j}H(X^{(n),2}_{\tau_{j}^2}, I_{j-1}, I_j)\bigg | \F_{\sigma} \bigg] \Bigg\},
\end{split}
 \end{equation}
 and the stopping times
 \begin{equation}
\tau^3_i \triangleq \tau_i^1 1_{A}+\tau_i^21_{\Omega-A}, \quad i\in \{1,\cdots,n\}.
 \end{equation}
 Then  $(\tau^3_1,\cdots,\tau^3_n) \in \S_{\sigma}^n$ and
 {\small
 \begin{equation}
 \begin{split}
 &\ME\left[\int_{\sigma}^{\tau_{c,d}}e^{-\alpha s}f(X^{(n),3 }_s,I^{(n),3}_s)ds -\sum_{j=1}^n
e^{-\alpha \tau_j}H(X^{(n),3}_{\tau_{j}^3}, I_{j-1}, I_j)\bigg | \F_{\sigma} \right]=\max  \Bigg\{\ME\bigg[\int_{\sigma}^{\tau_{c,d}}e^{-\alpha s}f(X^{(n),1}_s,I^{(n),1}_s)ds \\ &-\sum_{j=1}^n
e^{-\alpha \tau_j}H(X^{(n),1}_{\tau_{j}^1}, I_{j-1}, I_j)\bigg | \F_{\sigma} \bigg] + \ME\bigg[\int_{\sigma}^{\tau_{c,d}}e^{-\alpha s}f(X^{(n),2 }_s,I^{(n),2}_s)ds -\sum_{j=1}^n
e^{-\alpha \tau_j}H(X^{(n),2}_{\tau_{j}^2}, I_{j-1}, I_j)\bigg | \F_{\sigma} \bigg]\Bigg\},
\end{split}
 \end{equation}}
 and therefore $\mathcal{ Z}$ is directed upwards.

\noindent \textbf{Step 2.}
In this step we will show that
\begin{equation}\label{eq:ess-sup-Z-sigma}
Z_{\sigma}^{(n)}=\text{ess sup}_{\tau \in S_{\sigma}^{1}}\ME\left[\int_{\sigma}^{\tau}e^{-\alpha s}f(X^{(0)}_s,I_0)ds-e^{-\alpha \tau}\left(H(X^{(0)}_{\tau},I_0, I_1)+Z^{(n-1)}_{\tau}\right)\bigg | \F_{\sigma}\right].
\end{equation}
Let us fix $\tau_1 \in \S^1_{\sigma}$.
It follows from Step 1 that there exists a sequence  $\{\left(\tau_2^{k},\cdots,\tau_n^k\right)\}_{k \in \mathbb{N}} \in S_{\tau_1}^{n-1}$ such that
\begin{equation}
Z^{(n-1)}_{\tau_1}=\lim_{k \rightarrow \infty} \uparrow \ME\left[\int_{\tau_1}^{\tau_{c,d}}e^{-\alpha s}f(X^{(n-1),k }_s,I^{(n-1),k}_s)ds -\sum_{j=2}^n
e^{-\alpha \tau_j^k}H(X^{(n-1),k}_{\tau_{j}^k}, I_{j-1}, I_j)\bigg | \F_{\tau_1} \right],
\end{equation}
 Here, $X^{(n-1),k}$ is the solution of (\ref{eq:sde}) when we replace $I$ by $I^{(n-1),k}$ which is defined as
 \begin{equation}
 I^{(n-1),k}(t) \triangleq I_1 1_{\{t<\tau^k_2\}}+ \cdots+I_{n-1} 1_{\{\tau_{n-1}^k \leq t <\tau_{n}^k\}}+I_{n} 1_{\{t \geq \tau_{n}^k\}}.
 \end{equation}
 For every $k \in \N$, we have that $\left(\tau_1,\tau_2^{k},\cdots,\tau_n^{k}\right)\in S_{\sigma}^n$, and that
 \begin{equation}
 Z^{(n)}_{\sigma} \geq \limsup_{k \rightarrow \infty} \ME\left[\int_{\sigma}^{\tau_{c,d}}e^{-\alpha s}f(X^{(n),k }_s,I^{(n),k}_s)ds -\sum_{j=1}^n
e^{-\alpha \tau_j}H(X^{(n),k}_{\tau_{j}^k}, I_{j-1}, I_j)\bigg | \F_{\sigma} \right]
 \end{equation}
in which we take $\tau_1^k=\tau_1$ and $X^{(n),k}$ is the solution of (\ref{eq:sde}) when we replace $I$ by $I^{(n),k}$ which is defined as
 \begin{equation}
 I^{(n),k}(t) \triangleq I_0 1_{\{t<\tau_1\}}+I_1 1_{\{\tau_1 \leq t< \tau_2^k\}}+ \cdots+I_{n-1} 1_{\{\tau_{n-1}^k \leq t <\tau_{n}^k\}}+I_{n} 1_{\{t \geq \tau_{n}^k\}}.
 \end{equation}
We can then write
\begin{equation}\label{eq:ess-sup-half}
\begin{split}
 &Z^{(n)}_{\sigma} \geq \limsup_{k \rightarrow \infty} \Bigg\{\ME\left[\int_{\sigma}^{\tau_1}e^{-\alpha s}f(X^{(n),k }_s,I^{(n),k}_s)ds-e^{-\alpha \tau_1}H(X^{(n)}_{\tau_1},I_0, I_1)\bigg | \F_{\sigma}\right]
 \\ &\qquad \qquad +\ME\Bigg[\int_{\tau_1}^{\tau_{c,d}}e^{-\alpha s}f(X^{(n),k }_s,I^{(n),k}_s)ds-\sum_{j=2}^n
e^{-\alpha \tau^k_j}H(X^{(n),k}_{\tau_{j}^k}, I_{j-1}, I_j)\bigg | \F_{\sigma}\Bigg]\Bigg\}
\\ &= \ME\left[\int_{\sigma}^{\tau_1}e^{-\alpha s}f(X^{(0)}_s,I_0)ds-e^{-\alpha \tau_1}H(X^{(0)}_{\tau_1}, I_0, I_1)\bigg | \F_{\sigma}\right]
\\&+\ME\Bigg[\lim_{k \rightarrow \infty}\E^{X^{(0)}_{\tau_1},I_1}\bigg[\int_{\tau_1}^{\tau_{c,d}}e^{-\alpha s}f(X^{(n-1),k }_s,I^{(n-1),k}_s)ds-\sum_{j=2}^n
e^{-\alpha \tau^k_j}H(X^{(n-1),k}_{\tau_{j}^k}, I_{j-1}, I_j)\bigg]\Bigg | \F_{\sigma}\Bigg]
\\&=\ME\left[\int_{\sigma}^{\tau_1}e^{-\alpha s}f(X^{(0)}_s,I_0)ds-e^{-\alpha \tau_1}\left(H(X^{(0)}_{\tau_1},I_0, I_1)+Z_{\tau_1}^{(n-1)}\right)\bigg | \F_{\sigma}\right].
 \end{split}
 \end{equation}
Here, the first equality follows from the Monotone Convergence Theorem (here we used the boundedness assumption on $H$, see (\ref{eq:ass-H})). Since $\tau_1$ is arbitrary this implies that the left-hand-side of (\ref{eq:ess-sup-half}) is greater than the right-hand-side of (\ref{eq:ess-sup-Z-sigma}). Let us now try to show the reverse inequality.
Let for any $(\tau_1, \cdots, \tau_n) \in \S_{\sigma}^{n}$ let $I^{(n)}$ be given by (\ref{eq:I}) and let $X^{(n)}$ be the solution of (\ref{eq:sde}) when $I$ is replaced by $I^{(n)}$. And let us define $I^{(n-1)}$ by
 \begin{equation}
 I^{(n-1)}(t) \triangleq I_1 1_{\{t<\tau_2\}}+ \cdots+I_{n-1} 1_{\{\tau_{n-1} \leq t <\tau_{n}\}}+I_{n} 1_{\{t \geq \tau_{n}\}},
 \end{equation}
and let $X^{(n-1)}$ be the solution of (\ref{eq:sde}) when $I$ is replaced by $I^{(n-1)}$. Then
\begin{equation}
    \begin{split}
 \ME& \left[\int_{\sigma}^{\tau_1}e^{-\alpha s}f(X^{(n) }_s,I^{(n)}_s)ds-e^{-\alpha \tau_1}H(X^{(n)}_{\tau_1},I_0, I_1)\bigg | \F_{\sigma}\right]
 \\ &\qquad \qquad +\ME\left[\int_{\tau_1}^{\tau_{c,d}}e^{-\alpha s}f(X^{(n)}_s,I^{(n)}_s)ds-\sum_{j=2}^n
e^{-\alpha \tau_j}H(X^{(n)}_{\tau_{j}}, I_{j-1}, I_j)\bigg | \F_{\sigma}\right]
\\ &= \ME\left[\int_{\sigma}^{\tau_1}e^{-\alpha s}f(X^{(0)}_s,I_0)ds-e^{-\alpha \tau_1}H(X^{(0)}_{\tau_1},I_0, I_1)\bigg | \F_{\sigma}\right]
\\&+\ME\Bigg[\ME\bigg[\int_{\tau_1}^{\tau_{c,d}}e^{-\alpha s}f(X^{(n-1)}_s,I^{(n-1)}_s)ds-\sum_{j=2}^n
e^{-\alpha \tau_j}H(X^{(n-1)}_{\tau_{j}}, I_{j-1}, I_j)\bigg|\F_{\tau_1}\bigg]\Bigg | \F_{\sigma}\Bigg]
\\& \leq \ME\left[\int_{\sigma}^{\tau_1}e^{-\alpha s}f(X^{(0)}_s,I_0)ds-e^{-\alpha \tau_1}\left(H(X^{(0)}_{\tau_1}, I_0, I_1)+Z_{\tau_1}^{(n-1)}\right)\bigg | \F_{\sigma}\right],
 \end{split}
\end{equation}
now taking the essential supremum on the right-hand-side over all the sequences in $\S_{\sigma}^n$ we establish the desired inequality. Our proof in this step can be contrasted with the approach of Hamad\'{e}ne and Jeanblanc \cite{HJ2004} which uses the recently developed theory of Reflected Backward Stochastic Differential Equations to establish a similar result. The proof method we use above is more direct. On the other hand, as pointed out on page 14 of Carmona and Ludkovski \cite{CL2005}, it may be difficult to generalize the method of Hamad\'{e}ne and Jeanblanc \cite{HJ2004} to the cases when there are more than two regimes.

\noindent \textbf{Step 3.}

In this step we will argue that
\begin{equation}\label{eq:markv-cs}
Z^{(n)}_t = e^{-\alpha t} q^{(n)}(X^{(n)}_t,I^{(n)}_t), \quad t \geq 0,
\end{equation}
in which $I_t^{(0)}=I_0$, $t \geq 0$ and that $q^{(n)}$ is continuous in the $x$-variable.
We will carry out the proof using induction. First, let us write $q^{(1)}$ as
\begin{equation}\label{eq:q-1-x-i}
\begin{split}
q^{(1)}(x,i)&=\sup_{\tau \in \S_{0}^{1}} \ME\left[\int_0^{\tau_{c,d}}e^{-\alpha s}f(X_{s}^{(1)},I^{(1)}_s)-e^{-\alpha \tau}H(X_{\tau}^{(0)},I_0,I_1)\right]
\\ &=\sup_{\tau \in \S_{0}^{1}}\ME \left[\int_0^{\tau}e^{-\alpha s}f(X_{s}^{(0)},I_0)ds+\int_{\tau}^{\tau_{c,d}}e^{-\alpha s}f(X_{s}^{(1)},I_1)ds
-e^{-\alpha \tau}H(X_{\tau}^{(1)},I_0,I_1)\right]
\\&=\sup_{\tau \in \S_{0}^{1}}\ME \left[\int_0^{\tau}e^{-\alpha s}f(X_{s}^{(0)},I_0)ds+
\E^{X^{(0)}_{\tau},I_1}\left[\int_{\tau}^{\tau_{c,d}}e^{-\alpha s}f(X_{s}^{(0)},I_1)ds
-e^{-\alpha \tau}H(X_{\tau}^{(0)},I_0,I_1)\right]\right]
\\&=q^{(0)}(x,i)+ \sup_{\tau \in \S_{0}^{1}}\ME \left[e^{-\alpha \tau}\left(-q^{(0)}(X^{(0)}_{\tau},I_0)+q^{(0)}(X^{(0)}_{\tau},I_1)-H(X_{\tau}^{(0)},I_0,I_1)\right)\right].
\end{split}
\end{equation}
Let $\theta$ be the shift operator .
The third inequality in (\ref{eq:q-1-x-i}) follows from the strong Markov property of $(X^{(0)}_s)_{s \geq 0}$ and $(X^{(1)}_s,I^{(1)}_s)_{s \geq 0}$ and the fact that
\begin{equation}
\tau_{c,d}=\tau+\tau_{c,d} \circ \theta_{\tau},
\end{equation}
for any $\tau \in \S_{0}^1$, using which we can write
\begin{equation}
\ME \left[\int_0^{\tau}e^{-\alpha s}f(X_s^{(0)},I_0)ds\right]=q^{(0)}(x,i)-\ME \left[e^{-\alpha \tau}q^{(0)}(X^{(0)}_{\tau},I_0)\right],
\end{equation}
and
\begin{equation}\label{eq:imp-smp}
\begin{split}
&\ME\left[\int_{\tau}^{\tau_{c,d}}e^{-\alpha s}f(X_{s}^{(1)},I_1)ds\right]
=\ME \left[e^{-\alpha \tau} \ME\left[\int_{0}^{\tau_{c,d}}e^{-\alpha s}f(X_{s}^{(0)},I_1)ds\bigg| \F_{\tau} \right]\right]
\\&=\ME \left[e^{-\alpha \tau} \E^{X_{\tau}^{(0)},I_1}\left[\int_{0}^{\tau_{c,d}}e^{-\alpha s}f(X_{s}^{(0)},I_1)ds\right]\right]
=\ME\left[e^{-\alpha \tau}q^{(0)}(X^{(0)}_{\tau},I_1)\right]
\end{split}
\end{equation}

It is well known in the optimal stopping theory that (\ref{eq:markv-cs}) holds for $n=1$, if
\begin{equation}\label{eq:cond-fak}
A \triangleq \ME \left[\sup_{t \geq 0} e^{-\alpha t}\left(-q^{(0)}(X^{(0)}_{t},I_0)+q^{(0)}(X^{(0)}_{t},I_1)-H(X_{t}^{(0)},I_0,I_1)\right)^{-}\right]<\infty,
\end{equation}
and
\begin{equation}
x \rightarrow -q^{(0)}(x,I_0)+q^{(0)}(x,I_1)-H(x,I_0,I_1), \; x \in (c,d) \quad \text{is continuous},
\end{equation}
see Theorem 1 of Fakeev \cite{Fakeev}. (Fakeev requires $C_0$ continuity of
$-q^{(0)}(X^{(0)}_{t},I_0)+q^{(0)}(X^{(0)}_{t},I_1)-H(X_{t}^{(0)},I_0,I_1)$.
But this requirement is readily satisfied in our case since $X^{(0)}$ is continuous and since $-q^{(0)}(X^{(0)}_{t},I_0)+q^{(0)}(X^{(0)}_{t},I_1)-H(X_{t}^{(0)},I_0,I_1)$ is continuous, by the continuity assumption of $H$ and \eqref{eq:R-lr}.

But the growth conditions (\ref{eq:ass-H}), (\ref{eq:ass-f}) guarantee that (\ref{eq:cond-fak}) holds (using (\ref{eq:f-condition})).

 Now let us assume that (\ref{eq:markv-cs}) when $n$  is replaced by $n-1$ and that $q^{(n-1)}$ is continuous in the $x$-variable and show that (\ref{eq:markv-cs}) holds and $q^{(n)}$ is continuous in the $x$-variable.
>From Step 2 and the induction hypothesis we can write $q^{(n)}$ as
\begin{equation}
\begin{split}
q^{(n)}(x,i)&=\text{sup}_{\tau \in S_{0}^{1}}\ME\left[\int_{0}^{\tau}e^{-\alpha s}f(X^{(0)}_s,I_0)ds-e^{-\alpha \tau}H(X^{(0)}_{\tau},I_0, I_1)+Z^{(n-1)}_{\tau}\right]
\\&=\text{sup}_{\tau \in S_{0}^{1}}\ME\left[\int_{0}^{\tau}e^{-\alpha s}f(X^{(0)}_s,I_0)ds-e^{-\alpha \tau}\left(H(X^{(0)}_{\tau},I_0, I_1)+q^{(n-1)}(X^{(n-1)}_{\tau},I^{(n-1)}_{\tau})\right)\right]
\\&=\text{sup}_{\tau \in S_{0}^{1}}\ME\left[\int_{0}^{\tau}e^{-\alpha s}f(X^{(0)}_s,I_0)ds-e^{-\alpha \tau}\left(H(X^{(0)}_{\tau},I_0, I_1)+q^{(n-1)}(X^{(0)}_{\tau},I^{(0)}_{\tau})\right)\right],
\\&=q^{(0)}(x,i)+ \sup_{\tau \in \S_{0}^{1}}\ME \left[e^{-\alpha \tau}\left(-q^{(0)}(X^{(0)}_{\tau},I_0)+q^{(n-1)}(X^{(0)}_{\tau},I_1)-H(X_{\tau}^{(0)},I_0,I_1)\right)\right]
\end{split}
\end{equation}
where the third equality follows since $X^{(n-1)}_t=X^{(0)}_t$ for $t \leq \tau$, and the last equality can be derived using the strong Markov property of $(X^{(0)})_{t \geq 0}$ and $(X^{(n-1)}_t,I^{(n-1)}_t)_{t \geq 0}$.
The functions $H$ and $q^{(0)}$ are continuous in the $x$-variable and $q^{(n-1)}$ is assumed to satisfy the same property. On the other hand, we have that
\begin{equation}\label{eq:cond-fak-2}
 B\triangleq \ME \left[\sup_{t \geq 0} e^{-\alpha t}\left(-q^{(0)}(X^{(0)}_{t},I_0)+q^{(n-1)}(X^{(0)}_{t},I_1)-H(X_{t}^{(0)},I_0,I_1)\right)^{-}\right]<\infty,
\end{equation}
satisfies $B \leq A<\infty$, in which $A$ is defined in (\ref{eq:cond-fak}), since $(q^{(n)})_{n \in \N}$ is an increasing sequence of functions. Therefore, Theorem 1 of Fakeev \cite{Fakeev} implies that (\ref{eq:markv-cs}) holds. On the other hand, Lemma 4.2, Proposition 5.6 and Proposition 5.13 of \cite{DK2003} guarantee that $q^{(n)}$ is continuous. This concludes our induction argument and hence Step 3.

\noindent \textbf{Step 4.}
In this step we will show that the statement of the lemma holds using the results proved in the previous steps.

By definition we already have that
\begin{equation}
q^{(0)}(x,i)=w^{(0)}(x,i).
\end{equation}

Let us assume that the statement holds for $n$ replaced by $n-1$. From the previous step and the induction hypothesis we have that
\begin{equation}
\begin{split}
q^{(n)}(x,i)&=q^{(0)}(x,i)+ \sup_{\tau \in \S_{0}^{1}}\ME \left[e^{-\alpha \tau}\left(-q^{(0)}(X^{(0)}_{\tau},I_0)+q^{(n-1)}(X^{(0)}_{\tau},I_1)-H(X_{\tau}^{(0)},I_0,I_1)\right)\right]
\\&=q^{(0)}(x,i)+ \sup_{\tau \in \S_{0}^{1}}\ME \left[e^{-\alpha \tau}\left(-q^{(0)}(X^{(0)}_{\tau},I_0)+w^{(n-1)}(X^{(0)}_{\tau},I_1)-H(X_{\tau}^{(0)},I_0,I_1)\right)\right]
\\&= \sup_{\tau \in \S_{0}^{1}}\ME\left[\int_0^\tau
e^{-\alpha s}f(X^{(0)}_s,i)ds + e^{-\alpha\tau}
\left(w^{(n-1)}(X^{(0)}_{\tau},1-i)-H(X^{(0)}_{\tau}, i, 1-i)\right)\right]=w^{(n)}(x,i),
\end{split}
\end{equation}
where the last equality follows from (\ref{eq:imp-smp}). This completes the proof.

\bibliographystyle{dcu}
\bibliography{references}

\end{document}